\documentclass[a4paper,reqno]{amsart}

\usepackage{graphicx,pdfsync,amssymb, latexsym,amsfonts,amsbsy,color,subcaption,esint,mathtools,enumerate,nicefrac}
\usepackage{hyperref}
\hypersetup{
colorlinks=true,
linkcolor=blue,
filecolor=blue,      
urlcolor=blue,
citecolor=blue,
}

\usepackage[inline]{enumitem} 
\usepackage[utf8]{inputenc}
\usepackage[T1]{fontenc}

\usepackage{enumitem}
\setitemize{itemsep=0.5em}
\setenumerate{itemsep=0.5em}

\usepackage[top=1.1in, bottom=1in, left=1in, right=1in]{geometry}

\DeclareMathOperator{\Supp }{supp}
\DeclareMathOperator{\Id }{Id} 
 
\DeclareMathOperator{\D}{div}

\newtheorem{theorem}{Theorem}[section]
\newtheorem{lemma}[theorem]{Lemma}
\newtheorem{proposition}[theorem]{Proposition}
\newtheorem{example}[theorem]{Example}
\newtheorem{definition}[theorem]{Definition}
\newtheorem{corollary}[theorem]{Corollary}
\newtheorem{remark}[theorem]{Remark}

\usepackage{times}
\usepackage{setspace}

\usepackage{scalerel,stackengine}
\stackMath
\newcommand\wh[1]{%
\savestack{\tmpbox}{\stretchto{%
  \scaleto{%
    \scalerel*[\widthof{\ensuremath{#1}}]{\kern-.6pt\bigwedge\kern-.6pt}%
    {\rule[-\textheight/2]{1ex}{\textheight}}
  }{\textheight}%
}{0.5ex}}%
\stackon[1pt]{#1}{\tmpbox}%
}

\def \TT  {\mathbb{T}} 
\def \RR {\mathbb{R}}  
\def \NN {\mathbb{N}}  
\def \ZZ {\mathbb{Z}}  

\def \ep {\varepsilon}

\def \F {{\bf \Phi}}
\def \k {\kappa}

\def \L {\Lambda}

\def \X {\xi}

\def \ep {\varepsilon}

\newcommand{\comment}[1]{}

\numberwithin{equation}{section}

\begin{document}

\title[anomalous work for the NSE]{Anomalous dissipation, anomalous work, and energy balance for the Navier-Stokes equations}

\author{Alexey Cheskidov}

\address{Department of Mathematics, Statistics and Computer Science,
    University of Illinois At Chicago, Chicago, Illinois 60607}

\email{acheskid@uic.edu}

\author{Xiaoyutao Luo}

\address{Department of Mathematics, Duke University, Durham, NC 27708}

\email{xiaoyutao.luo@duke.edu}


\date{\today}

\begin{abstract}
In this paper, we study the energy balance for a class of solutions of the Navier-Stokes equations with external forces in dimensions three and above. The solution and force are smooth on $(0,T)$ and the total dissipation and work of the force are both finite. We show that a possible failure of the energy balance stems from two effects. The first is the \emph{anomalous dissipation} of the solution, which has been studied in many contexts. The second is what we call the \emph{anomalous work} done by the force, a phenomenon that has not been analyzed before. There are numerous examples of solutions exhibiting \emph{anomalous work}, for which even a continuous energy profile does not rule out the anomalous dissipation, but only implies the balance of the strengths of these two effects, which we confirm in explicit constructions. More importantly, we show that there exist solutions exhibiting \emph{anomalous dissipation} with zero \emph{anomalous work}. Hence the violation of the energy balance results from the nonlinearity of the solution instead of artifacts of the force. Such examples exist in the class $u \in L_t^{3 } B^{\nicefrac{1}{3} -}_{3,\infty}$ and $f \in L_t^{2-} H^{-1}$, which implies the sharpness of many existing conditions on the energy balance. 

\end{abstract}

\maketitle

\section{Introduction}
The Navier-Stoke equations (NSE) for the incompressible viscous fluids are
\begin{equation}\label{eq:NSE}\tag{NSE}
\begin{cases}
\partial_t u -\nu\Delta u + \D(u \otimes u) + \nabla p =f &\\
\D u =0,
\end{cases}
\end{equation}
where $u(x,t)$ is the unknown $d$-dimensional velocity field, $p(x,t)$ is the scalar pressure, and $f(x,t)$ is the  external force. We consider the equations on the $d$-dimensional torus $\TT^d$ for $d\geq 3$ with normalized viscosity coefficient $\nu =1$, but all the results can be extended to $\RR^d$ as well.

In this paper, we study the validity of the energy equality of the forced NSE. The goal is twofold: on one hand, to identify possible causes for failure of energy balance; and on the other hand, to construct counterexamples showing the sharpness of positive results. For instance, we prove that the regularity of solutions possessing anomalous dissipation can be on the borderline of Onsager's critical spaces, where the energy equality holds.

\subsection{Background and previous works}
If a solution of the NSE is regular enough, then the change of the energy is equal to the work done by the force minus the total energy dissipation. This can be seen easily by multiplying the equation \eqref{eq:NSE} by $u$ and integrating in space-time
\begin{equation}\label{eq:energy_equality}
\frac{1}{2}\|u  (t_1)\|_2^2 -\frac{1}{2}\|u  (t_0)\|_2^2  + \int_{t_0}^{t_1} \|\nabla u  (t)\|_2^2 \,dt =\int_{t_0}^{t_1}   \langle u  , f  \rangle \  \, dt . 
\end{equation}

However, there is not enough regularity to justify this formal computation for weak solutions, and it is expected that some weak solutions may not obey the energy balance. It was in fact conjectured by Onsager \cite{MR0036116} that in $3$-dimensional inviscid flows, solutions with H\"older continuity $\alpha>\nicefrac{1}{3}$ conserve energy, and the conservation of energy may fail if $\alpha<\nicefrac{1}{3} $. Such a  failure of the energy conservation is often called anomalous dissipation as it is due to the lack of smoothness of the velocity rather than the viscous dissipation. Onsager's conjecture for the Euler equations has been the topic of recent research activities and is generally considered solved in both directions. See the works of Eyink \cite{MR1302409}, Constantin-E-Titi \cite{MR1298949}, Duchon-Robert \cite{MR1734632} and Constantin-Cheskidov-Friedlander-Shvydkoy \cite{MR2422377} for the positive direction and the works of De Lellis-Szekelydihi \cite{MR3090182,MR3254331}, Buckmaster-De Lellis-Isett-Szekelyhidi \cite{MR3374958} and Isett \cite{MR3866888,1706.01549} for the negative direction.

In the context of the NSE, the existence/nonexistence of anomalous dissipation is still an open question. Pioneering results of Leray state that for any finite energy initial data there exists at least one weak solution satisfying the energy inequality when there is no force, or more generally when the force $f\in L_t^2 H^{-1}_x$. Weak solutions obeying the energy inequality are called Leray-Hopf solutions. The question of whether such solutions satisfy the energy equality \eqref{eq:energy_equality} is open, and only conditional criteria are available so far in the unforced case. Notably, Lions \cite{MR0115018} proved that $u \in L^4_t L_x^4$ implies the energy balance in 3D,  which was extended to $u \in L^p_t L_x^q$ for $\frac{2}{p} + \frac{2}{q} =1$ in all dimensions by Shinbrot \cite{MR0435629}. These two classical results can be recovered by the estimates in \cite{MR2422377} and interpolations. Recent works of Leslie and Shvydkoy \cite{MR3759872,MR3842051} prove local-in-space results and energy balance for a certain Type-I blowup of strong solutions at the first blowup time. Shortly later, the two authors of current paper \cite{1802.05785} obtained weak-in-time improvement for Shinbrot's condition showing that $u \in L^{p,w}_t L^q_x$ for $\frac{2}{p} + \frac{2}{q} =1$ is enough for energy balance.

In contrast to the Euler equations, there are no known examples of anomalous dissipation for the unforced NSE. There are constructions of wild solutions though, with arbitrary smooth energy profile by Buckmaster and Vicol \cite{MR3898708}, wild solutions with some regularity in time by Buckmaster, Colombo and Vicol \cite{1809.00600} and by the two authors \cite{2009.06596}. However, the energy dissipation is infinite for all these wild solutions, i.e. they are not in the $L^2_t H^1_x$ class. Even though one can still identify forward or backward energy cascades in some of those solutions, the energy balance equation does not make sense outside of the $L^2_t H^1_x$ class.
For more results using the method of convex integration in fluid dynamics, see for example \cite{MR3479065,1710.11186,MR3951691,MR3884855,1812.11311,1812.08734,1907.10436} and references therein.

\subsection{The setup}
To find genuine examples of anomalous dissipation in the viscous setting, here we consider the simplest scenario: a possible violation of the energy balance at one point, while the solutions and forces under consideration are smooth on $(0,T)$~\footnote{There are technically two singular points $t\to 0^+$ and $t\to T^-$. For consideration of energy balance, we focus on $t\to T^-$.} (see Section \ref{sec:def_solutions} for precise definitions). We do not assume $f \in L^2_t H^{-1}$, but merely finite work done by the force. Such a relaxation is natural considering that a finite work done is the minimal assumption for which the energy balance equation makes sense, and in fact, the energy equality \eqref{eq:energy_equality} holds automatically on $(0,T)$ for this class of smooth solutions.

One of our goals is to determine whether the energy balance still holds in these one-point singularity scenarios. A heuristic in mind is that as the time approaches the possible singularity, a fixed amount of energy may move to the infinite wave-number creating a jump in the energy profile, cf. the energy jump formula \eqref{eq:energy_jump_formula}. If such a heuristic is viable, it could lead to a surprising result: a smooth solution on $(0,T)$ that does not obey the energy balance on $[0,T]$. Note that all previous nonconservative Euler solutions in \cite{MR3090182,MR3254331,MR3374958,MR3866888} have a different mechanism from our one-point model.

It turns out that besides the anomalous dissipation, there is another possible cause of the failure of the energy balance, which had not been studied in the past. The forces considered in this paper are smooth on $(0,T)$ and produce a finite energy input, but may not necessarily be in the class $f \in L^2_t H^{-1}$. This allows for a possibility of a new phenomenon, what we call the \textit{anomalous work}, which is elaborated below. 

Denote by $u_{\leq q}$ the Littewood-Paley projection onto the frequencies $\lesssim 2^q$. Then the energy balance for $u_{\leq q}$ reads
\begin{equation}\label{eq:intro_energy_qshell}
\frac{1}{2}\|u_{\leq q} (t)\|_2^2 -\frac{1}{2}\|u_{\leq q} (t_0)\|_2^2 + \int_{t_0}^{t} \|\nabla u_{\leq q} (\tau)\|_2^2 \,d \tau = -\int_{t_0}^{t} \langle ( u\cdot \nabla u )_{\leq q}, u_{\leq q} \rangle \, d\tau +\int_{t_0}^{t}  \langle u_{\leq  q} , f_{\leq q} \rangle    \, d\tau . 
\end{equation}
Since the solution has regularity $u \in L^\infty L^2 \cap L^2 H^1$, the terms on the left side of \eqref{eq:intro_energy_qshell} will converge to their natural limits as $q \to \infty$ and only the terms on the right might cause the failure of the energy balance.

On one hand, as in \cite{MR2422377}, one says the \textit{anomalous dissipation} occurs (over some time interval) if 
$$
\int \langle ( u\cdot \nabla u )_{\leq q}, u_{\leq q} \rangle \, d\tau \not \to 0 \quad \text{as $q\to \infty$}.
$$
This motivates the following quantitative definition of the anomalous dissipation $\overline{\Pi}>0$ considered in the literature
\begin{equation}\label{eq:intro_def_anomalous_dissipation}
\overline{\Pi}  = \limsup_{q\to \infty } \Big| \int_{T/2}^{T} \langle ( u\cdot \nabla u )_{\leq q}, u_{\leq q} \rangle \, dt \Big|  .
\end{equation}
On the other hand, since the solution $u\in L^2_t H^1$, if the force $f   \in L^2_t H^{-1}$, then by the duality between $L^2_t H^1$ and $ L^2_t H^{-1}$, the last term in \eqref{eq:intro_energy_qshell} converges to its natural limit, the total work done by the force. However, if $f   \not \in L^2_t H^{-1}$, then in general (on some time interval) it could happen that
\begin{equation}\label{eq:intro_anomalous_work}
\int  \langle u_{\leq  q} , f_{\leq q} \rangle    \, d\tau  \not \to \int  \langle u  , f  \rangle    \, d\tau \quad \text{as $q\to \infty$}
\end{equation}
and when \eqref{eq:intro_anomalous_work} occurs, we say the force has \textit{anomalous work}. 

Similar to the anomalous dissipation and motivated by \eqref{eq:intro_anomalous_work}, we can quantify the anomalous work $\overline{\F} $ of the force by
\begin{align}\label{eq:intro_def_anomalous_work}
\overline{\F}  & = \limsup_{q\to \infty } \Big| \int_{T/2}^{  T} \big( \langle  u_{\leq q} \cdot f_{\leq q} \rangle -\langle u\cdot f \rangle \big) \, dt \Big|.
\end{align}
This definition of anomalous work is consistent with \eqref{eq:intro_energy_qshell} and in line with the classical definition of anomalous dissipation \eqref{eq:intro_def_anomalous_dissipation}.

Essentially, by allowing ``rougher'' forces, high-high interactions of the solution and the force can result in nonzero anomalous work, which in turn can violate the energy balance without any energy cascade.
Such examples are abundant:  take a smooth stationary Euler flow and push its frequency to infinity by a suitable force as $t\to T^-$ will produce a finite dissipation example with a jump in the energy at $t=T $. However, such examples, considered in Section \ref{sec:trivial_positive}, are not very interesting and more or less trivial.

Our main goal is to contract examples with genuine energy cascade, where the anomalous dissipation is nonzero. First, in Section \ref{sec:zeroF_nonzeroPi}, we will focus on the class of solutions with no anomalous work, which we argue are the physical ones, and were all the classical conditions for the energy balance, such as $u \in  L^{p}_t L^{q}_x$ for some $\frac{2}{p} +\frac{2}{q} =1$, still hold. In this class, we will construct examples of solutions violating the energy balance, and show the sharpness of positive results. Second, in Section \ref{sec:continuousE}, we will construct an example where the energy equality holds, but both anomalous work and anomalous dissipation are nonzero, in effect canceling each other.

We first present the main results to classify possible scenarios where the anomalous dissipation or anomalous work persists.

\subsection{Main results}
Let us recall several key notations. Throughout the paper, we denote $Q_T =\TT^d \times (0,T)$ the space-time domain, and $  \mathcal{N}(Q_T) =\{(u,f)\}$ the class of smooth solutions with smooth force and finite energy input introduced in Definition \ref{def:smooth_solution}. Note that even though the solutions $(u,f)$ are smooth on $Q_T$, they must blow up near $t=0$ or $t=T$ to develop nontrivial anomalous dissipation or anomalous work. The reader should note that the notion of smoothness in this paper refers only to the interior of the time interval $(0,T)$.

Our positive results state as follows.
\begin{theorem}\label{thm:main_result_Pi_F_equal}
Suppose $(u,f) \in \mathcal{N}(Q_T)$, then the following two conditions are equivalent.
\begin{enumerate}
    \item The energy is continuous: $u \in C([0,T ];L^2)$.
    \item $(u,f)$ satisfies the energy equality \eqref{eq:energy_equality} on $[0,T]$.
\end{enumerate}

If one of the conditions holds, then the anomalous dissipation and the anomalous work are the same: $\overline{\Pi} = \overline{\F}$.
\end{theorem}
Therefore, continuous energy only implies the same strength of anomalous dissipation and anomalous work, which is different than in the unforced case or when $f \in L^2_t H^{-1}$. In particular, Theorem \ref{thm:main_result_Pi_F_equal} recovers a special case of the following result.
\begin{theorem}\label{thm:main_result_Pi_F_0}
Suppose $(u,f) \in \mathcal{N}(Q_T)$ such that $f \in L^p(0,T ;H^{-2+\frac{2}{p}})$ for some $1\leq p \leq 2$. Then the anomalous work $\overline{\F}=0$. Moreover, if in addition $u \in C([0,T ];L^2)$, then $\overline{\Pi} = \overline{\F}=0$.
\end{theorem}

It is also worth noting that Theorem \ref{thm:main_result_Pi_F_0} is sharp, as will show that there are counterexamples when forcing $f \in L^{2-\ep}_t H^{-1}$ for any $\ep>0$, see Theorem \ref{thm:main_result_continuous_E_with_Pi}. We also remark that $u\in  L^3_t B^{\frac{1}{3}   }_{3,\infty} $ implies zero anomalous dissipation $\overline{\Pi} =0$ as in the classical unforced settings, see Lemma \ref{thm:onsager}.

Our next result concerns the uniqueness problem for solution class $\mathcal{N}(Q_T)$. It is known that the classical Ladyzhenskaya-Prodi-Serrin uniqueness criterion also holds in the forced case for $f\in L^2_t  H^{-1}$. Note that there are many refinements over the classical uniqueness results, notably \cite{MR2237686} by Germain, \cite{MR2838337} by Chemin and \cite{MR3767658} by Barker. It is worth mentioning that the latter two results \cite{MR2838337,MR3767658}  do not apply to our setting since in our case the initial data $u_0 \in L^2$ and there is no uniform regularity assumption for $f $. It seems that the result of \cite{MR2237686} can be extended to our setting (see Remark 3.3, \cite{MR2237686}), however, we choose not to do so, avoiding technicalities in harmonic analysis. We show that assuming only finite energy input, at least the classical uniqueness results hold. 

\begin{theorem}\label{thm:main_result_LPS}
Let $(u ,f), (v,f)  \in \mathcal{N}(Q_T)$ with the same force $f$ and initial data $u (0) =v(0) =u_0$. If in addition $u  \in L^p_t L^q_x(Q_T)$ with $\frac{2}{p} + \frac{d}{q} = 1$ and $v$ is continuous in $L^2$ at $t=0$, then $u = v$.
\end{theorem}

The result of Theorem \ref{thm:main_result_LPS} is expected considering that under such assumptions it can be seen as a weak-strong uniqueness result for ``Leray-Hopf weak solutions with finite energy input'' in the spirit of Definition \ref{def:smooth_solution}.

The next two theorems are our most surprising results. We construct solutions whose anomalous work is zero while anomalous dissipation is not. In this case, the violation of energy balance stems from the solution itself rather than the force. More importantly, these solutions can be made arbitrarily close to the borderline spaces of energy balance, cf. \cite{MR0115018,MR0435629,MR2422377,1802.05785}. 

\begin{theorem}\label{thm:main_result_Pi_without_F}
For any $\ep>0$ and dimension $d\geq 3$, there exists a  solution $(u,f)\in \mathcal{N}(Q_T)$ such that $ (u,f)$ satisfies the energy equality on $[0,T)$, namely
\begin{equation}
\frac{1}{2}\|u  (t_1)\|_2^2 -\frac{1}{2}\|u  (t_0)\|_2^2  + \int_{t_0}^{t_1} \|\nabla u  (t)\|_2^2 \,dt =\int_{t_0}^{t_1}   \langle u  , f  \rangle \  \, dt \quad \text{for all $t_0,t_1 \in [0,T)$} ,
\end{equation}
the anomalous work vanishes
\begin{equation}
\lim_{q\to \infty} \int_{T/2}^{  T} \big( \langle  u_{\leq q} \cdot f_{\leq q} \rangle -\langle u\cdot f \rangle \big) \, dt =0,
\end{equation}
and the energy is jump discontinuous at $t=T$,
\begin{align}
\lim_{t\to T^-} \|   u(t) \|_2^2 > \|   u(T) \|_2^2.
\end{align}

Moreover, $u$ is almost Onsager critical: $u \in L_t^{3 } B^{\frac{1}{3} -\ep }_{3,\infty} \cap L^{p}_t L^{q}_x$ for any $\frac{2}{p} +\frac{2}{q} =1+\ep $ and $f$ is almost Leray-Hopf: $f \in L^{2-\ep}_t H^{-1 }$.
\end{theorem}

Now we note that in the class of zero anomalous work solutions $\overline \F=0$, classical positive results still hold: if $u \in L_t^{3 } B^{\frac{1}{3}}_{3,\infty}$ or $u \in  L^{p}_t L^{q}_x$ for some $\frac{2}{p} +\frac{2}{q} =1$, then the energy equality is satisfied. Indeed, if $u$ is such a solution, then Theorem~\ref{thm:onsager} (cf. \cite{MR2422377}) implies no anomalous dissipation and the 
energy jump formula \ref{eq:energy_jump_formula} implies that $u$ is continuous in $L^2$ at $t=T$ and hence the energy equality thanks to Theorem~\ref{thm:main_result_Pi_F_equal}.
These positive results do not contain any conditions on the force, so the force by itself cannot create anomalous dissipation, which can only come from the energy cascade when $\overline \F=0$. Now thanks to Theorem~\ref{thm:main_result_Pi_without_F}, we know that these positive results are sharp. 

\begin{corollary}
Consider the class of solutions in $\mathcal{N}(Q_T)$ with zero anomalous work:
\[
\mathcal{S} = \{(u,f) \in \mathcal{N}(Q_T): \overline \F =0 \}.
\]
\begin{enumerate}
    \item If $(u,f) \in \mathcal{S}$ is such that  $u \in L_t^{3 } B^{\frac{1}{3}}_{3,\infty}$ or $u \in  L^{p}_t L^{q}_x$ for some $\frac{2}{p} +\frac{2}{q} =1$, then the energy equality is satisfied on $(0,T]$.
    \item For any $\ep>0$, there exists $(u,f) \in \mathcal{S}$ such that $u \in L_t^{3 } B^{\frac{1}{3} -\ep }_{3,\infty} \cap L^{p}_t L^{q}_x$ for any $\frac{2}{p} +\frac{2}{q} =1+\ep $ and the energy equality is not satisfied on $(0,T]$.
\end{enumerate}
\end{corollary}

Our last result shows that Theorem \ref{thm:main_result_Pi_F_equal} and Theorem \ref{thm:main_result_Pi_F_0} are sharp in another way. In the class of $\mathcal{N}(Q_T)$, continuous energy does not rule out anomalous dissipation due to the presence of anomalous work. These solutions are also arbitrarily close to the borderline spaces of energy balance. 

\begin{theorem}\label{thm:main_result_continuous_E_with_Pi} 
For any $\ep>0$ and dimension $d\geq 3$, there exists $(u,f) \in \mathcal{N}(Q_T)$ such that  $ ( u,f)  $ satisfies energy equality on $[0,T]$
\begin{equation}
\frac{1}{2}\|u  (t_1)\|_2^2 -\frac{1}{2}\|u  (t_0)\|_2^2  + \int_{t_0}^{t_1} \|\nabla u  (t)\|_2^2 \,dt =\int_{t_0}^{t_1}   \langle u  , f  \rangle \  \, dt \quad \text{for all $t_0,t_1 \in [0,T]$},
\end{equation}
but the anomalous dissipation still occurs
\begin{equation} \label{intro-positiveflux}
\lim_{q\to \infty } \Big| \int_{T/2}^{T} \langle ( u\cdot \nabla u )_{\leq q}, u_{\leq q} \rangle \, dt \Big| > 0.
\end{equation}

In addition, $u$ is almost Onsager critical: $u \in L_t^{3 } B^{\frac{1}{3} -\ep }_{3,\infty} \cap  L^{p}_t L^{q}_x$ for any $\frac{2}{p} +\frac{2}{q} =1+\ep $, and $f$ is almost Leray-Hopf: $f \in L^{2-\ep}_t H^{-1 }$.
\end{theorem}

\subsection{Strategy of the proof}
One of the main ingredients in the proof is a construction of a sequence of building blocks, vector fields with the optimal energy flux, which we then glue together in time. We use two different gluing mechanisms to achieve the positivity of the energy flux \eqref{intro-positiveflux}  for our solutions. In Theorem \ref{thm:main_result_Pi_without_F} all the energy escapes to the infinite wavenumber as time $t\to T^-$, resulting in a jump of the energy. This solution enjoys the anomalous dissipation with no anomalous work. On the other hand, in Theorem \ref{thm:main_result_continuous_E_with_Pi}, the energy does not completely transfer to the next shell as the force injects energy at each frequency producing the anomalous work (in addition to the regular work). This solution encounters both anomalous dissipation and anomalous work that balance each other resulting in the energy equality. The designed vector fields are intermittent of dimension close to $d-2$ based on the heuristics in \cite{1802.05785}, so that we also achieve finite dissipation. In Theorem \ref{thm:main_result_Pi_without_F} the gluing in time is more delicate and the time scales are carefully designed so that the force has no anomalous work. Our method of constructing such pathological solutions is very flexible and less restrictive than traditional methods since no uniform regularity of the force is assumed.

In view of the energy cascade in the construction of Theorem \ref{thm:main_result_Pi_without_F}, the force is designed so that it does not interfere with the energy cascade but only helps the solution to keep the desired structure. This can be seen as an implementation of a blow-up in Tao's averaged NSE \cite{MR3486169} to the actual NSE, but with a force. The delay mechanism is enforced via gluing building blocks with a positive flux. So at each time, the transfer of energy occurs only at one particular scale. The intermittency of building blocks combined with the delay mechanism results in the optimal energy cascade to high modes, which allows us to construct examples with finite dissipation and almost critical Onsager's regularity.

\subsection{Physical motivation} 
The effect of rough forces on the energy cascade in turbulent flows has been extensively studied experimentally as well as numerically \cite{vassilicos2001intermittency,PhysRevE.67.066306, mazzi_vassilicos_2004}. Experimental setups usually involve a stirrer with a self-similar structure or a grid of a fractal dimension above two. In fractal forced direct numerical simulations, rough forces are used to model fluid obstacles, such as grids. See for instance \cite{mazzi_vassilicos_2004}, where the roughness of the force well surpassed the Leray-Hopf $L^2_tH^{-1}_x$ class for obstacles of high fractal dimension. It has been shown that the energy spectrum of fractal forced turbulence deviates from Kolmogorov's law and the energy dissipation rate does not exhibit Kolmogorov scaling, but instead diverges at high Reynolds numbers, which was also confirmed theoretically with rigorous upper bounds in \cite{doi:10.1063/1.2425101}.

Even though there is evidence that fractal obstacles should be modeled by rough forces, one can argue that such forces should not exhibit any anomalous work. For instance, in the case of an elastic fractal grid we would expect the force to satisfy
\[
\lim_{q\to \infty} \int_{t_1}^{t_2} \langle  u_{\leq q} \cdot F^{\text{fractal}}_{\leq q} \rangle \, dt =0,
\]
on sufficiently large time intervals $[t_1,t_2]$, so that the total work vanishes in the strong physical sense.  This combined with a driving (low mode, or Leray-Hopf) force $F^{\text{LH}}$ in the experiment would result in $\overline\F =0$, were the anomalous work $\overline \F$ is as in \eqref{eq:intro_def_anomalous_work} computed for $f=F^{\text{fractal}}+F^{\text{LH}}$ over one interval since we only study the case of a single time singularity.

As we assume zero anomalous work, i.e., no nonphysical mechanism to produce some pathological work on small scales, we study whether solutions to the \eqref{eq:NSE} can exhibit anomalous dissipation. Physically this would correspond to a dissipation mechanism for high velocities on small scales where the equations are not valid anymore. Such small scales cant get engaged if the energy cascade is stronger than the dissipation.

\subsection{Concluding comments}
The zero anomalous work construction of Theorem \ref{thm:main_result_Pi_without_F} is \emph{not} specific to any particular Littlewood-Paley decomposition: the anomalous work of the force is zero for any dyadic Littlewood-Paley cutoff such that each projection $\Delta_q$ has frequencies in the range $[ 2^{q-1},   2^{q+1} ]$. In addition, the value of the anomalous dissipation  \eqref{eq:intro_def_anomalous_dissipation} is always equal to the energy jump at $t=T$, so it represents the energy balance defect, as in the unforced case.


It is also worth noting that if one can strengthen Theorem \ref{thm:main_result_Pi_without_F} to $f \in L^2_t H^{-1}$, then a suitable modification of the proof would imply the nonuniqueness of Leray-Hopf weak solutions. Indeed, standard methods of constructing Leray-Hopf weak solutions can be used to remove the jump, and thus create a new solution.

We finish our discussion by comparing our method of constructing solutions with convex integration that has been developed for fluid dynamics in recent years. It is possible that by using convex integration and allowing for forcing one can also construct solutions with zero anomalous work and nonzero anomalous dissipation. However, it seems that with current techniques this is impossible in 3D and only works in very high dimensions. There is at least one advantage of using convex integration: forcing given by convex integration solutions has small low frequencies. So the force can be made arbitrarily small in $L^\infty_t  W^{-1,1}$ or $L^\infty_t W^{-k,p}$ for sufficiently large $k$ depending on $p\in[1,\infty]$~\footnote{It seems also possible to reach $f \in L_t^2 H^{-1-\ep}$ for any $\ep>0$ by convex integration, which is different than the ones $f \in L_t^{2-\ep} H^{-1}$ obtained here.}. The examples given in this paper do not have such a property.

\subsection*{Organization of the paper}
The rest of the paper is divided into the following sections. 
\begin{itemize}
    \item In Section \ref{sec:def_solutions} we introduce the solution class $\mathcal{N}(Q_T)$ and discuss basic properties and its relationship with other notion of solution. In particular, Theorem \ref{thm:main_result_LPS} is proved in Section \ref{subsec:LPS_unique}.
    \item In Section \ref{sec:def_Pi_F} we formulate anomalous dissipation and anomalous work by a Littlewood--Paley decomposition.
    \item Section \ref{sec:trivial_positive} is devoted to both the positive results and the counterexamples. On one hand, we prove Theorem \ref{thm:main_result_Pi_F_equal} and Theorem \ref{thm:main_result_Pi_F_0}. On the other hand, we give simple examples of nonzero anomalous work.
    \item The last two sections are dedicated to constructing counterexamples. Using intermittent vector fields with optimal energy flux, we prove Theorem \ref{thm:main_result_Pi_without_F} in Section \ref{sec:zeroF_nonzeroPi} and Theorem \ref{thm:main_result_continuous_E_with_Pi} in Section \ref{sec:continuousE}.
\end{itemize}

\section{Smooth solutions with finite input}\label{sec:def_solutions}
\subsection{Functional spaces and notations}
Throughout the paper, we consider the \eqref{eq:NSE} on $\Omega=\TT^d$ for $d\geq 3$. The space of test functions is denoted by $ \mathcal{D}(\Omega)$ or simply $\mathcal{D} $. Respectively, the space of distributions is denoted by $\mathcal{D}'$. The Lebesgue norm is written as $\| \cdot \|_{L^p(\Omega)} = \| \cdot \|_{ p } $

For any $f\in \mathcal{D}'$, its Fourier transform is denoted as $\widehat{f} $ or $\mathcal{F}f$. For any distribution $u \in \mathcal{D} $, we will use the Littlewood-Paley decomposition
\begin{equation}
u = \sum_{q \geq -1}\Delta_q u
\end{equation}
where $\Delta_q  $ denotes the projection onto frequencies $\sim\lambda_q :=2^q$. Note that throughout the paper $\lambda_q =2^q$ is the dyadic number for any $q \in \ZZ$.

The notation $X \lesssim Y$ means $ X \leq C Y$ for some constant $C>0$. If the constant $C$ depends on certain parameters $a_0,a_1 \dots a_n$, then we write $X \lesssim_{a_0,a_0 \dots a_n}$. $X \gtrsim Y$ is defined similarly and $X\sim Y$ means $X \lesssim Y$ and $Y \lesssim X$ at the same time.

For any $s\in \RR$, the Sobolev spaces $H^{s}$ consists of all distributions satisfying
\begin{align}
\sum_{q \geq -1} \lambda_q^2 \|\Delta_q u\|_2^2 <\infty.
\end{align}

For any Banach space $X$, a function $f:[0,T]\to X$ is called weakly continuous into $X$ if $\langle f,g \rangle_{X} : [0,T] \to \RR$ is continuous for any $g\in X'$. The space $C_w (0,T;X)$ consists of all functions weakly continuous into $X$.  


The space $\mathcal{L}(0,T)$ consists of all $f \in L_{loc}^{1}$ such that the limit
$$
\lim_{a\to 0^+, b\to T^-}\int_{a}^{b} f \, dt\quad \text{exists. }
$$ 
In particular, the Lebesgue space $L^1$ is included, $L^{1}(0,T) \subset \mathcal{L}(0,T)$. The space $\mathcal{L}(0,T)$ will be used to characterize finite work done by the force.

In what follows, unless otherwise indicated, the notation $\langle \cdot, \cdot \rangle $ is reserved for the standard $L^2$ inner product.

\subsection{Smooth solutions with finite input}
We now begin to introduce our class of smooth solutions $\mathcal{N}(Q_T)$. Since we are interested in the energy balance law, the solution should be of finite energy and the input from the force should also be finite.  
\begin{definition}[Solutions class $\mathcal{N}(Q_T)$]\label{def:smooth_solution}
Let $T>0$ and the space-time domain $Q_T := \TT^d \times (0,T)  $. We say $(u,f)$ is a \textbf{smooth solution with finite energy input}, or $(u,f) \in \mathcal{N}(Q_T) $ for simplicity, if all of the followings are satisfied
\begin{itemize}
    \item $u \in C^\infty_{t,x}(Q_T) \cap L^\infty(0,T; L^2(\Omega))$, $f\in C^\infty(Q_T)$;
    \item $u$ and $f$ satisfies \eqref{eq:NSE} for any $(t,x) \in Q_T$;
    \item $ \langle u,f \rangle \in \mathcal{L}(0,T  )$;  
    \item The limits of $f(t)$ exist in $\mathcal{D}'$ as $t \to 0^+$ and $t\to  T^- $.
\end{itemize}

\end{definition}

\begin{remark}
Here are a few comments concerning the solution class $( u, f ) \in  \mathcal{N}(Q_T) $.

\begin{itemize}
    \item As discussed in the introduction, we only assume the smoothness of $(u,f)$ on $(0,T)$, not $[0,T]$. The only uniform regularity we have is $u \in L^\infty(0,T; L^2(\Omega))$.

    \item The assumption that $f $ can be extended to a continuous distribution on $[0,T]$ is mainly for the compatibility of a weak solution. Without such an assumption, $u$ can not be extended to a weak solution on $[0,T]$, see Lemma \ref{lemma:solutions_extended_0_T}.
    
    \item The finite energy input condition $ \langle u,f \rangle \in \mathcal{L}(0,T  )$ is the minimal assumption for which the energy inequality makes sense on $ ( 0,T]$. This includes the usual case $f \in L^2 H^{-1}$ of the Leray-Hopf solutions.
\end{itemize}

\end{remark}

\begin{remark}
There are two potential singular points for solutions in $\mathcal{N}(Q_T)$: $t\to 0^+$ and $t\to T^-$. The latter case $t\to T^-$ is where we study possible failure of energy balance while $t\to 0^+$ is where possible non-uniqueness may emerge. Note that failure of energy balance at $t\to 0^+$ is the worst-case scenario of nonuniqueness, whereas the possible non-unique solutions considered in \cite{MR3341963} satisfy the energy balance at $t\to 0^+$.
\end{remark}

The following lemma shows that possible discontinuities of energy profile of such solutions can only be jumps.
\begin{lemma}\label{lemma:fancyLebesgue}
Assume that $(u,f)$ is a smooth solution of \eqref{eq:NSE} on $Q_T$ and $u \in L^2(0,T;H^1)$. Then $(u,f) \in \mathcal{N}(Q_T) $ if and only if $\displaystyle \lim_{s\to t}\|u(s)\|_2^2$ exists for $t=0^+$ and $T^-$.
\end{lemma}
\begin{proof}
Indeed, since $(u,f)$ is smooth on $(0,T)$, the energy equality holds:
\[
\frac{1}{2}\|u(t)\|_2^2 = \frac{1}{2}\|u(t_0)\|_2^2 -\int_{t_0}^t \|\nabla u\|_2^2 \, d\tau + \int_{t_0}^t \langle u, f \rangle \, d\tau, \qquad 0<t_0< t<T.
\]
Since $u \in L^2 H^1$, the limits of $\|u(t)\|_2^2 $ at $0^+$ and $T^-$ exist if and only if $\langle u, f \rangle  \in \mathcal{L}(0,T) $.
\end{proof}

The next theorem shows that all the terms in the energy balance equation are finite on the closed interval $[0,T]$.
\begin{theorem}\label{thm:finite_input_imply_finite_dissip}
If $(u,f) \in \mathcal{N}(Q_T)$, then the energy dissipation is finite: $u \in L^2(0,T;H^1)$.
\end{theorem}
\begin{proof}
Again, this is a simple consequence of the energy equality for $u(t)$ on $(0,T)$, according to which the energy dissipation is given by 
\[
\int_{t_0}^t \|\nabla u\|_2^2 \, d\tau =   \frac{1}{2}\|u(t_0)\|_2^2 - \frac{1}{2}\|u(t)\|_2^2 + \int_{t_0}^t \langle u, f \rangle \, d\tau, \qquad 0<t_0< t<T.
\]
Using $u\in  L^\infty(0,T;L^2)$ and the definition of $\mathcal{N}(Q_T)$, we can take the limsup as $t_0  \to 0^-$ and $t \to T^+$ of both sides to obtain the desired result.
\end{proof}

\subsection{Relation to weak solutions}
The aim here is to compare our class of solutions $\mathcal{N}(Q_T)$ with weak and Leray-Hopf solutions. First, recall that a weak solution of the NSE on $[0,T]$ is a weakly continuous and weakly divergence-free vector field $u\in C_w([0,T];L^2)$ solving \eqref{eq:NSE} in the sense of distributions.

\begin{lemma}\label{lemma:solutions_extended_0_T}
If $(u,f) \in \mathcal{N}(Q_T)$, then weak limits of $u(t)$ in $L^2$ exist as $t\to 0^+$ and $t\to T^-$.
\end{lemma}
\begin{proof}
Thanks to the incompressibility, it suffices to consider divergence-free test functions. By the weak formulation of the NSE, for $0<t<t_1<T$ and divergence-free $\varphi \in \mathcal{D}$ we have
\begin{align}
\langle u(t ), \varphi \rangle= \langle u(t_1), \varphi \rangle -\int_{t}^{t_1} \langle u  , \Delta \varphi \rangle \, d\tau + \int_{t}^{t_1} \langle u\otimes u    :\nabla \varphi \rangle \, d\tau -\int_{t}^{t_1} \langle f  ,  \varphi \rangle \, d\tau.
\end{align}
Sending $t \to 0^+$ or $t \to T^-$, we need to show that the limits of all terms on the right-hand side exist. This easily follows from our finite energy assumption $L^\infty_t L^2$ of $u$ and the assumption that the limits of $f$ as a distribution exists (see Definition \ref{def:smooth_solution}). 
\end{proof}

Thanks to Lemma \ref{lemma:solutions_extended_0_T} we can extend $u(t)$ to $[0,T]$ by weak continuity in $L^2$. So $u(0)$ and $u(T)$ will always denote weak limits $u(t)$ as $t\to 0+$ and $t\to T-$ respectively. Since the extended $u(t)$ is weakly continuous on $[0,T]$, it is a weak solution of \eqref{eq:NSE} on $[0,T]$.

Weak solutions satisfying the energy inequality are called Leray-Hopf weak solutions. We can show that solutions in $ \mathcal{N}(Q_T)$ are Leray-Hopf solutions for positive times. 

\begin{lemma}
If $(u,f) \in \mathcal{N}(Q_T)$, then for any $\ep>0$, $(u,f)$ is a Leray-Hopf weak solution on $[\ep,T]$, namely the energy inequality
\begin{align}
\frac{1}{2}\|u(t) \|_2^2 +  \int_{t_0}^{t} \|\nabla u \|_2^2 \, d\tau \leq \frac{1}{2}\|u(t_0) \|_2^2 +  \int_{t_0}^{t}  \langle u,f \rangle  \, d\tau 
\end{align}
is satisfied for all $t_0, t\in[\ep,T]$, $t\geq t_0$.
\end{lemma}
\begin{proof}
This immediately follows from the fact that the energy equality is satisfied on $(0,T)$, and the lower semi-continuity of $\|u\|_2$ at $t\to T^-$ which is a consequence of weak continuity of $u(t)$.
\end{proof}

\begin{remark}
Here, by Leray-Hopf weak solution we mean a weak solution $u \in C_w([0,T];L^2) \cap L^2_t H^1(Q_T) $ satisfying the energy inequality, without any assumption on the regularity of the force on $[0,T]$.
\end{remark}
\begin{remark}
Note that in general $(u,f) \in \mathcal{N}(Q_T)$ is not a Leray-Hopf weak solution on $[0,T]$. The energy may be discontinuous at $t=0$ which is forbidden if energy inequality is satisfied starting at $t_0=0$. This is the reason that in Theorem \ref{thm:main_result_LPS} we need to assume the $L^2$-continuity at $t=0$.
\end{remark}

\subsection{Uniqueness results for $\mathcal{N}(Q_T)$}\label{subsec:LPS_unique}
In the last part of this section, we briefly touch on the uniqueness result in the class $\mathcal{N}(Q_T)$. We follow classical strategies of proving uniqueness. Suppose we have two smooth solutions $(u,f), (v,f) \in \mathcal{N}(Q_T)$ for the same force $f$, and $v$ is a ``strong solution''. Let $w:= u-v$ be the difference, the energy space $\mathcal{E}_T$ be
$$
\mathcal{E}_T:= L^\infty_t L^2 \cap L^2_t H^1(Q_T)
$$ 
and the trilinear operator $\mathcal{T }$ be
\begin{align}
\mathcal{T }:(a,b,c) \to \int_{0}^{T}\langle a\cdot \nabla b, c \rangle\, dt.
\end{align}

There are mainly two key points in the classical argument. First, we need the energy inequality on $[0,T]$ for both solutions. Second, we need one strong solution in some path space $\mathcal{P}_T \subset L^1_{loc}(Q_T)$ such that the trilinear operator $\mathcal{T }$ is continuous on $\mathcal{E}_T \times \mathcal{E}_T \times \mathcal{P}_T$, namely
\begin{align}\label{eq:boundedness_T}
\big| \mathcal{T }(a,b,c) \big| \lesssim \|a \|_{\mathcal{E}_T } \|b \|_{\mathcal{E}_T } \| c\|_{\mathcal{P}_T}.
\end{align}

With these in hand, one can then use a continuity argument and Proposition \ref{prop:L2_stability_w} to show that $\| w\|_{\mathcal{E}_T} \leq 0$, hence uniqueness.

For $\mathcal{P}_T= L^p_t L^q(Q_T) $, $\frac{2}{p}+ \frac{d}{q}=1$, the estimate \eqref{eq:boundedness_T} is classical for the 3D NSE (see for example \cite{Prodi1959}), it is also standard to extend it to other dimensions. To prove Theorem \ref{thm:main_result_LPS} we only need to justify the following $L^2$ stability estimate of the difference $w$.
\begin{proposition}\label{prop:L2_stability_w}
Under the assumptions of Theorem \ref{thm:main_result_LPS}, the following equality holds
\begin{equation*}
\|w (t)\|_2^2+ 2\int_{0}^{t} \| \nabla w(\tau) \|_2^2\, d\tau = 2 \int_0^{t} \langle w\cdot \nabla w, v\rangle \, d\tau \quad \text{for all $t\in[0,T]$}.
\end{equation*}
\end{proposition}
\begin{proof}
Since $u$ and $v$ are smooth on $Q_T$, the difference $w$ satisfies
\begin{equation}\label{eq:w_equation}
\partial_t w -\Delta w +  w \cdot \nabla w + v\cdot \nabla w+ w\cdot\nabla v +\nabla \pi =0 \quad \text{for all $0<t<T$} .
\end{equation}
Multiplying \eqref{eq:w_equation} by $w$ and integrating, we get
\begin{align*}
\|w (t)\|_2^2-\|w (t_0)\|_2^2+ 2\int_{t_0}^{t} \| \nabla w(\tau) \|^2_2  \, d\tau = 2 \int_{t_0}^{t} \langle w\cdot \nabla w, v \rangle \, d\tau \quad \text{for all $t_0,t \in(0,T)$}.
\end{align*}
Since $u  \in L^p_t L^q_x(Q_T)$ with $\frac{2}{p} + \frac{d}{q} = 1$ and $v$ is bounded in $L^2$ at $t=0$, the trilinear term make sense when $ t_0 \to 0^+$, namely
$$
\int_{ 0}^{t} \langle w\cdot \nabla w, v \rangle \, d\tau <\infty .
$$
Moreover, the space $L^p_t L^q_x(Q_T)$ is Onsager subcritical (see Theorem~\ref{thm:onsager}), so $u$ satisfies the energy balance and is continuous in $L^2$ at $t=0$. Since $v$ is also continuous in $L^2$ at $t=0$ by the standing assumption, we have $ \| w(t_0)\|_2^2 \to 0$ as $t_0 \to 0^+$, which concludes the proof.

\end{proof} 
\section{Anomalous dissipation and anomalous work}\label{sec:def_Pi_F}
In this section, we formulate the concept of anomalous dissipation and anomalous work through the Littlewood--Paley decomposition. As briefly discussed in the introduction, these two quantities arise naturally when studying the energy balance of the NSE. Note that for our solution class $\mathcal{N}(Q_T)$, these two quantities do not depend on the particular decomposition used here, which is the same as in the unforced case or $ f\in L^2_t H^{-1}$, see the energy jump formula \eqref{eq:energy_jump_formula} below. 

\subsection{Littlewood--Paley decomposition}\label{subsec:LP_def}
We briefly introduce a Littlewood--Paley decomposition on $\TT^d$. Throughout the paper we use the notation $\lambda_q = 2^q$ for all $q\in \ZZ$. Let radially non-increasing $\chi(\xi) \in C^\infty_c(\RR^d)$ be such that $\chi(\xi)=1$ for $|\xi| \leq 1/2$ and $\chi=0$ for $|\xi| \geq 1$. Let $\varphi(\xi) = \chi(\lambda_1^{-1} \xi) - \chi( \xi) $ and define for any $q\in \NN$ $\chi_q(\xi)=\chi (\lambda_q^{-1} \xi)$  and $\varphi_q(\xi)=\varphi (\lambda_q^{-1} \xi)$. Then we have the partition
\begin{equation}
\chi(\xi) + \sum_{q\geq 0} \varphi_q(\xi) = 1.
\end{equation} 

We then let $\Delta_q$ be the Littlewood--Paley projection with symbol $ \varphi_q(\xi)$ for $q\geq 0$ or $\chi(\xi) $ if $q= -1$. For any  distribution $u \in \mathcal{D}(\TT^d)$ one has
\begin{equation*}
u = \sum_{q \geq -1}\Delta_q u
\end{equation*}
in the sense of distribution. We also use notations $\Delta_{\leq q} := \sum_{r \leq q } \Delta_r$, $u_q: = \Delta_q u$ and $ u_{\leq q} := \Delta_{\leq q}  u$. From the telescoping identity
\begin{equation}\label{eq:LP_chi_q}
\chi(\xi) + \sum_{ 0 \leq r\leq q} \varphi_r(\xi) = \chi_{q+1}(\xi),
\end{equation} 
it follows that $\Delta_{\leq q}$ is the Littlewood--Paley projection with symbol $\chi_{q+1}(\xi)$.

Let us recall that for any $s\in \RR$, $p,q \in [1,\infty]$, the (in-homogeneous) Besov space $B^{s}_{p,q}$ is equipped with the following norm
$$
\| u\|_{ B^{s}_{p,q}} : = \big\| \big( \lambda_q^{s} \| u_q\|_p  \big)_{q \geq -1}\big\|_{ \ell^q(\ZZ)}.
$$
It is known that $B^{s}_{2,2} = H^s $.

We refer readers to \cite{MR2099035} for more background on harmonic analysis applying to fluid dynamics.

\subsection{Energy flux}
We will use the Littlewood--Paley decomposition to formulate our definitions of the anomalous dissipation and anomalous work for $(u,f) \in \mathcal{N}(Q_T)$. We focus on the endpoint $t\to T^-$ and start with the cutoff energy equality and let $(T-h,T) \subset (0,T)$.

Multiplying \eqref{eq:NSE} by $(u_{\leq q})_{\leq q }$ and integrating on $\TT^d \times (T-h,T )$ gives
\begin{align} \label{eq:energy_flux_LP}
\frac{1}{2}\|u_{\leq q} (t)\|_2^2\Big|_{t=T-h}^{t=T} + \int_{T-h}^{T} \|\nabla u_{\leq q} (s)\|_2^2 \,d\tau =-\int_{T-h}^{T} \langle  u\cdot \nabla u , (u_{\leq q})_{\leq q}\rangle + \langle u_{\leq  q} , f_{\leq q} \rangle \  \, d\tau .
\end{align}
Since $(u,f) \in \mathcal{N}(Q_T)$, we have $u \in L^2 H^1$ thanks to Theorem~\ref{thm:finite_input_imply_finite_dissip}. So the limit of the right-hand side as $q \to \infty$ always exists. However, as discussed in the introduction, the two energy flux terms on the right-hand side may not converge to the work $\int\langle u  , f   \rangle    \, dt  $, resulting a failure of energy balance. Let $\Pi_{q} (t ,h)$ and $\F_{q} (t ,h)$ be defined by 
\begin{align}\label{eq:def_Pi}
\Pi_{q}(t ,h) : =  \int_{[t-h,t+h]\cap [0,T]} \int_{\TT^d} (u\otimes u) :  \nabla(u_{\leq q})_{ \leq q } \, dx \, d\tau,
\end{align}
and respectively
\begin{align}\label{eq:def_F}
\F_{q} (t ,h)  : = \int_{[t-h,t+h]\cap [0,T]} \int_{\TT^d} \left(u_{\leq q} \cdot f_{\leq q} - u\cdot f \right) \, dx \, d\tau.
\end{align}

In what follows, we will simply refer to $\Pi_{q}$ and $ \F_{q}$ as flux terms. The energy balance through wavenumber $\lambda_q$ can then be written as
\begin{align}\label{eq:motivationfor_Pi_F}
\frac{1}{2}\|u_{\leq q} (t)\|_2^2\Big|_{t=T-h}^{t=T} + \int_{T-h}^{T} \|\nabla u_{\leq q} (s)\|_2^2 \,d\tau = \int_{T-h}^{T}   \langle u  , f  \rangle    \, d\tau + \Pi_q(T,h) + \F_q(T,h).
\end{align}

One particular usage of the two quantities $\Pi_{q}$ and $\F_{q} $ is to measure the possible jump discontinuity of the energy. Indeed, by taking a limit as $h\to 0^+$ and using the fact that $u \in L^2 H^1$, we obtain the formula for the energy jump at $t\to T^-$
\begin{align} \label{eq:energy_jump_formula}
\frac{1}{2}\|u(T)\|_2^2 - \lim_{t \to T -}\frac{1}{2}\|u(t)\|_2^2 =
\lim_{q \to \infty}  \left[ \Pi_q(T,h) + \F_q(T,h)\right].
\end{align}

In the literature (see \cite{MR2422377} for example), $\Pi_{q}$ is called the energy flux through wavenumber $\lambda_q $ which is used to capture the anomalous dissipation of the solution. In the unforced case, $\limsup| \Pi_{q}| =0$ immediately implies the energy equality. The conclusion also holds if we assume $f\in L^2_t H^{-1}$ in the forced case. 

However, for solutions in $\mathcal{N}(Q_T) $, the flux term $\F_{q}$ in \eqref{eq:motivationfor_Pi_F} may not converges to $0$ because we no longer have the bound $f\in L^2_t H^{-1}$. The failure of energy balance is due to the high-high interaction between the solution and the force. In Section \ref{sec:trivial_positive} we will show that many such examples can be obtained.

Based on the discussion above, it is natural to introduce the following.

\begin{definition}[Anomalous dissipation and anomalous work]\label{def:AD_AEI}
Let $(u,f) \in \mathcal{N}(Q_T)$. For any $t\in[0,T]$ the anomalous dissipation $\overline{\Pi}(t) $ at time $t$ is defined by 
\begin{align*}
\overline{\Pi} (t) =   \limsup_{q \to \infty} | \Pi_{q} (t ,h)|,
\end{align*}
and  the anomalous work $\overline{\F }(t) $ at time $t$  is defined by 
\begin{align*}
\overline{\F } (t) =   \limsup_{q \to \infty} | \F_{q} (t ,h)|,
\end{align*}
where  $\Pi_{q} (t ,h)$ and $\F_{q} (t ,h)$ are as in \eqref{eq:def_Pi} and \eqref{eq:def_F}. By a slight abuse of notation, we simply write $ \overline{\Pi}=\overline{\Pi} (T)  $ and $\overline{\F }  = \overline{\F } (T) $.
\end{definition}

Note that in the definition, we did not specify the value of the parameter $h>0$. The next lemma shows that any $T>h>0$ will give the same definition.
\begin{lemma}\label{lemma:welldefined}
The anomalous dissipation $\overline{\Pi}(t)$ and  the anomalous work $\overline{\F }(t)$ are well-defined. 
\end{lemma}
\begin{proof}
We need to show that $\limsup_{q \to \infty} | \Pi_{q} (t ,h)| = \limsup_{q \to \infty} | \Pi_{q} (t ,h')|$ for any $h,h' >0$, and mutatis mutandis for $\overline{\F }(t)$. We only show this for $\overline{\Pi }(T)$. Assuming $h'>h>0$, it suffices to show
\begin{equation}\label{eq:lemma10}
\left| \int_{T-h'}^{T-h} \int_{\TT^d} (u\cdot \nabla)u \cdot (u_{\leq q})_{\leq q} \, dx \, dt \right| \to 0, \qquad \text{as $q \to \infty$}
\end{equation}

Since $u$ is smooth on $(0,T)$, in particular this implies that 
$$
\|u \|_{C^1_{t,x}([T-h',T-h]\times \Omega)}  <\infty.
$$
It is easy to see that $u \in C^1_{t,x} \cap L^\infty_t L^2 \cap  L^2_t H^1$  implies \eqref{eq:lemma10}  immediately.

\end{proof}

It is straightforward to verify that  $\overline{\Pi}(t)=\overline{\F }(t)=0$ for any solution $(u,f) \in \mathcal{N}(Q_T)$ and any $t\in(0,T)$. In what follows we are only interested in the nontrivial case $t\to T^-$.

\section{Positive results and simple examples}\label{sec:trivial_positive}
The goal of this section is two-fold. On one hand, we provide positive results to classify possible scenarios of the violation of the energy balance. This is reflected by the loss of continuity of the energy. On the other hand, we provide simple examples with positive anomalous work.

\subsection{Positive results}
We prove several positive results in this subsection. In particular, these results would imply Theorem \ref{thm:main_result_Pi_F_equal} and thus Theorem \ref{thm:main_result_Pi_F_0}.

First, let us show that for the solution class $\mathcal{N}(Q_T)$, the continuity of energy is equivalent to the energy equality on $[0,T]$.
\begin{lemma}\label{lemma:equivalency}
Suppose $(u,f) \in \mathcal{N}(Q_T)$, then $u \in C([0,T]; L^2)$ if and only if $(u,f)$ satisfies the energy equality on $[0,T]$.
\end{lemma}
\begin{proof}
Since $u \in L^2 H^1 $ and $ \langle u,f\rangle \in \mathcal{L}(0,T)$, the sufficiency is easy. Let us show the necessity. In this case one can then use the energy equality on $(0,T)$ and pass to the limit thanks to $u \in C([0,T]; L^2)$.
\end{proof}

Next, we finish the proof of Theorem \ref{thm:main_result_Pi_F_equal}. In the unforced cases, it is expected that solutions exhibiting anomalous dissipation would have a discontinuous energy. However, when a force $f\not \in L^2_t H^{-1}$ is present, there might be cancellations between anomalous dissipation and anomalous work. Thus the continuity of the energy only implies that these two effects are of the same strength. 

\begin{lemma}\label{lemma:main_positive}
Let $(u,f) \in \mathcal{N}(Q_T)$. If $u \in C([0,T ];L^2)$, then $\overline{\Pi} = \overline{\F}$.
\end{lemma}
\begin{proof}
By the energy jump formula \eqref{eq:energy_jump_formula},
\[
\lim_{q \to \infty}  \left[ \Pi_q(T,h) + \F_q(T,h)\right] = \frac{1}{2}\|u(T)\|_2^2 - \lim_{t \to T -}\frac{1}{2}\|u(t)\|_2^2 = 0,
\]
due to the continuity of $u$ in $L^2$. This immediately implies $\overline{\Pi}= \overline{\F} $. 
\end{proof}

Note that by Lemma \ref{lemma:equivalency} and Lemma \ref{lemma:main_positive}, we have thus obtained Theorem \ref{thm:main_result_Pi_F_equal}.

When the force is in the Leray-Hopf class $f \in L^2(0,T ;H^{-1})$, there is no anomalous work. In addition, any subsequence of $ |\Pi_{q}|$ will converge to the anomalous dissipation $\overline{\Pi}$ in this case.
\begin{lemma}\label{lemma:unique_Pi}
Suppose $(u,f) \in \mathcal{N}(Q_T) $ is such that $f \in L^p(0,T ;H^{-2+ \frac{2}{p}})$ for some $1 \leq p \leq 2$. Then the anomalous work is zero: $\overline{\F} = 0 $.

Moreover, the anomalous dissipation satisfies
$$
\overline{\Pi}(t)=  \lim_{q \to \infty} | \Pi_{q} (t ,h)| \quad \text{for any $t\in[0,T]$.}
$$ 
\end{lemma}
\begin{proof}
Since $(u,f) \in \mathcal{N}(Q_T) $, we have $ u \in L^\infty_t L^2 \cap L^2_t H^1$, and hence $ u \in L^q_t H^\alpha$, where $q$ is the dual H\"older exponent of $p$ and $\alpha:=  2- \frac{2}{p}$. Since also $f\in L^p_t H^{-\alpha}$, by duality and interpolation, we have
\begin{align}
\left|\int_{T-h}^T \int_{\TT^d} \left(u_{\leq q} \cdot f_{\leq q} - u\cdot f \right) \, dx \, d\tau\right| &\leq \int_{T-h}^T \|u_{\geq q} \|_{H^{\alpha}} \|f_{\geq q} \|_{H^{-\alpha}} \,dt\\
&\lesssim     \|u_{\geq q} \|_{L^{q}_t H^{\alpha}} \|f_{\geq q} \|_{L^p_t H^{-\alpha}}\\
&\to 0,
\end{align}
as $q\to \infty$, and the anomalous work is zero $\overline{\F} = 0 $ by definition \eqref{eq:def_F}. 

As for the second claim, we only show this for $t=T$. Due to the energy jump formula \eqref{eq:energy_jump_formula}, the limit
\[
\lim_{q \to \infty}  \left[ \Pi_q(T,h) + \F_q(T,h)\right]
\]
exists. Since $\overline{\F} = 0$, we also have that $\lim_{q \to \infty} \F_q(T,h) =0$, and hence $\lim_{q \to \infty}|\Pi_q(T,h)|$ exists, and is equal to $\overline \Pi$.

\end{proof}

Finally, we note that $B^{1/3}_{3,\infty}$ is the Onsager's space as in the unforced setting. This follows immediately from the estimates on the flux in \cite{MR2422377} since the definition of the anomalous dissipation is the same.
\begin{theorem}[{\cite[Theorem 6.1]{MR2422377}}]\label{thm:onsager}
If $(u,f) \in \mathcal{N}(Q_T) $ and
$$
u \in L^3_t B^{1/3}_{3,\infty} ,
$$ then $\overline{\Pi}=0$.
\end{theorem}

We also quote the positive result of \cite{1802.05785} in our settings which improves upon the condition of Shinbrot \cite{MR0435629}. 
\begin{theorem}[{\cite[Theorem 1.1]{1802.05785}}]\label{theorem:weak_lbg}
Let $(u,f) \in \mathcal{N}(Q_T)$ and $1 \leq \beta<p \leq \infty$ such  that $\frac{2}{p} + \frac{1}{\beta}<1$. If~\footnote{Note that $L^{p,w}$ denotes the weak Lebesgue space.}
\begin{equation}\label{eq:weak_lbg}
u \in L^{\beta,w}_t B^{\frac{2}{\beta}+\frac{2}{p}-1}_{p,\infty}  ,
\end{equation}
then $\overline{\Pi}=0$.
\end{theorem}

\begin{remark}
Note that Theorem \ref{theorem:weak_lbg} requires finite dissipation while Theorem \ref{thm:onsager} holds true for any weak solutions $(u,f)$. Even though Theorem \ref{theorem:weak_lbg} relies on the flux estimates in \cite{MR2422377}, there is no mutual implication between these two conditions, see \cite[Section 5]{1802.05785} for details.
\end{remark}

\subsection{Trivial examples}
We now switch from the positive results to discussing various examples with anomalous work. As one shall see,  anomalous work is perhaps the simplest way to produce a jump discontinuity in the energy. 

\subsubsection{Oscillations}\label{eg:1}
The first example is based on shear flows glued together by a partition in time. We show that anomalous work can be caused by pure oscillations.

Given $T>0$, let us choose a sequence of smooth cutoffs $h_i \in C^\infty_c(\RR)$ such that 
$$
\Supp h_i \subset [T-  2^{-\beta (i-1)}T,T- 2^{-\beta (i+1)}T],
$$ 
and 
$$
\sum_{} h_i(t)^2=1 \quad \text{for $t \in [0,T)$.}
$$ 

\begin{example}
Let $k ,l \in \ZZ^d $ so that $k\cdot l =0$ and $\beta>2$. Consider the NSE on torus $\TT^d$ for $d\geq 3$ and define the solution
$$
u(x,t):= 2\sum_{i\geq 1} h_i(t)\sin(2^i l\cdot x)k,
$$ 
for $t\in[0,T)$ and the force
$$
f(t,x):=\partial_t u -\Delta u. 
$$
\end{example}

In this case, obviously $u $ and $f$ are smooth on $Q_T$. Due to the choice of time cutoffs we also have $\|u(t) \|_2^2 = 1$ for $t>0$, and hence $\frac{d}{dt}\|u(t) \|_2^2=0 $. Since $\beta>2$, it is easy to verify that $u \in L^2 H^1$ and due to the smoothness of $u$ and $f$ as well as the energy equality on $(0,T)$ we get 
\begin{align}
\langle u,f \rangle =  \| \nabla u(t) \|_2^2  \quad \text{for all }0<t <T,
\end{align}
which implies $\langle u,f \rangle \in \mathcal{L}(0,T) $ (in fact $L^1(0,T)$). It is also easy to see that weak $L^2$ limits of $u$ exists as $t\to 0^+, T^-$, and the latter one is zero. Therefore the energy is discontinuous at $t=T$. Also, the limits of $f(t)$ exist in $\mathcal{D}'$ as $t \to 0^+$ and $t\to  T^- $. Thus $(u,f) \in \mathcal{N}(Q_T)$. 

We claim that the jump discontinuity at $t=T$ is caused by the anomalous work. Indeed, it is not hard to see that $u \in L^\infty_{t,x}(Q_T)$ or even $ \D(u\otimes u) =0$, and hence the anomalous dissipation $\overline{\Pi }=0$ for this solution.

\subsubsection{Concentration}
The next example shows that the development of concentrations can also lead to anomalous work. Here we use the Concentrated Mikado flows used in \cite{MR3951691} which was based on the Mikado flows introduced by Daneri and Sz\'ekelyhidi, Jr. in \cite{MR3614753}.

Recall that for dimension $d \geq 3$ and $k\in \ZZ^d$ Concentrated Mikado flows $\mathbb{W}_k^{\mu} \in C_0^\infty(\TT^d ) $ are periodic pipe flows whose supported pipes have small radius of size $~\mu^{-1}$ with direction $k$; See Section 2 in \cite{MR3951691} for the exact definition. 

Let us fix some constant $0<\alpha<\frac{1}{2}$ and introduce a time-dependent  concentration parameter 
$$
\mu(t):= \frac{1}{(T-t)^\alpha}   .
$$

\begin{example} \label{eg:2}
Consider the NSE on $\TT^d$ for $d\geq 3$. Let $\mathbb{W}_k^{\mu} \in C_0^\infty(\TT^d ) $ be a Concentrated Mikado flow. Define the solution 
$$
u(x,t)= \mathbb{W}_k^{\mu(t)}
$$ 
and the forcing $f(x,t):=\partial_t u -\Delta u  $ on $(0,T)$.
\end{example}

Since $\D(u\otimes u)=0$ for all $t\in(0,T)$, $(u,f) \in \mathcal{N}(Q_T)$. As before,  $u$ and $f$ are smooth on $Q_T$ and the weak $L^2$ limit of $u(t)$ is zero as $t\to T^-$. So the energy is discontinuous at $t=T$. Note that due to the scaling property of the Concentrated Mikado flow, $\| \nabla v(t)\|_2^2 = c \mu^2$ for some fix constant $c$, which implies $u \in L^2(0,T;H^1)$ since $0<\alpha<\frac{1}{2}$. Note that $ \|u(t) \|_2^2$ is a constant on $(0,T)$, and due to the energy equality on $(0,T)$ we get 
\begin{align}
\langle u,f \rangle =  \| \nabla u(t) \|_2^2  \quad \text{for all }0<t <T,
\end{align}
which implies $\langle u,f \rangle \in \mathcal{L}(0,T) $. So $(u,f) \in \mathcal{N}(Q_T)$. 
It is obvious that the energy flux $\overline{\Pi }=0$ since $\D(u\otimes u)=0$. Therefore, the discontinuity of the energy must be caused by the anomalous work.

\subsubsection{Tao's blowup solution of averaged NSE} 
In a remarkable paper \cite{MR3486169}, Tao proved a finite time blowup of a smooth solution to an averaged NSE, which satisfies many classical harmonic analysis estimates of the original NSE. Here, we show that this solution is also in our solution class $\mathcal{N}(Q_T)$ for the original NSE.

Let $u: [0,T^*) \to C^\infty(\RR^3)$ be the blowup solution~\footnote{ One may notice that $u$ is only $H^{10}$ as stated in \cite{MR3486169}, however, smoothness can be obtained by bootstrapping as the initial data is Schwartz  and $\widetilde{B}$ satisfies the same estimates as $B$  on every Sobolev space $W^{s,p} $ for $1<p<\infty$. See the discussion after Remark 1.4 in \cite{MR3486169}.} to the averaged NSE in  \cite[Theorem 1.5]{MR3486169}. So, $u$ is a classical solution to $\partial_t u -\Delta u +\widetilde{B}(u,u) = 0$ that blows up at time $T^*$ with a Schwartz initial data $u_0$.

\begin{example}\label{eg:3}
Consider the forcing $f := B(u,u) - \widetilde{B}(u,u)$. Then $(u,f) \in \mathcal{N}(Q_{T^*})$ and $\langle u,f \rangle = 0$ for all $t \in [0,T^*)$.  
\end{example}

The weak $L^2$ limit  of $u$ exists as $t \to T^*$ by the weak formulation of the average NSE. Therefore by a standard argument, one can show that the force $f  = B(u,u) - \widetilde{B}(u,u)$ can also be extended to $t=T^*$ as a continuous distribution. All conditions in Definition \ref{def:smooth_solution} has been verified, and thus $(u,f) \in \mathcal{N}(Q_{T^*})$.

Since the averaged bilinear operator $\widetilde{B} $ still fulfills the orthogonality $\langle \widetilde{B}(u,v),v \rangle =0$, then by construction $\langle u,f \rangle = 0$ for all $t \in [0,T^*)$.

\begin{remark}
In the above example, the work is zero for $(u,f)$. However, at the moment it is not known whether the energy has a jump at $t\to T^-$ since only the blowup of higher Sobolev norms is shown~\footnote{It seems that the energy is continuous since the blowup portion of the solution has energy tending to zero according to \cite{Priv_Comm_Tao2019}.}, see \cite[Proposition 6.3 ]{MR3486169}. It is also not clear whether anomalous work $\overline{\F} =0$ or even anomalous dissipation $\overline{\Pi}=0$ as the solution $u$ was not designed to saturate the energy flux for the original NSE. 
\end{remark}

\subsubsection{Concentration with zero input}
In this last example, we show that one can also achieve zero work done in Example \ref{eg:2} by incorporating the decay of heat flow. Denote by $ v(x,t)$  the solution in Example \ref{eg:2}. Let $h \in C^\infty(\RR^+)$ be the unique solution of ODE
\begin{equation}\label{eq:ODE}
\frac{dh }{dt} = -2\| \nabla v(t)\|_2^2 h(t)
\end{equation}
with initial data $h(0)=1$. As before, $\| \nabla v(t)\|_2^2 = c \mu^2$ for some fix constant $c$. We thus obtain explicitly 
$$
h(t) = h(T)e^{ \frac{2c}{1-2\alpha}(T-t)^{1-2\alpha}}.
$$

\begin{example}
Consider the NSE on $\TT^d$ for $d\geq 3$ and define a solution 
\begin{align}
u(x,t):= h^{\frac{1}{2}} (t)v(x,t) \quad \text{and} \quad f:= \partial_t u -\Delta u,
\end{align}
for $t\in(0,T)$.
Then $(u,f) \in \mathcal{N}(Q_t)$, $\langle u ,f \rangle =0$ for all $t\in (0,T)$ and the energy has a jump at $t=T$.
\end{example}

Arguing as in Example \ref{eg:2}, $u\in L^2 H^1$ since $h(t)\leq 1$.
Since $h$ solves the ODE \eqref{eq:ODE},
$$
\frac{d}{dt}\|u \|_2^2 = -2\|\nabla u \|_2^2.
$$
This implies $\langle u ,f \rangle =0$ due to the energy equality on $(0,T)$.

The energy has a jump at $t=T$ as in Example \ref{eg:2} because $u(T)=0$, but the limit of the energy is positive as $h(T)>0$.

\section{Examples with zero anomalous work}\label{sec:zeroF_nonzeroPi}
In this section, we will finish the proof of Theorem \ref{thm:main_result_Pi_without_F}. The idea is to glue vector fields with large flux in a way that the leading cancellation in $\F_{n}$ at the level of energy results in zero anomalous work. The construction of vector fields is inspired in part by Eyink's example in \cite{MR1302409} and by the constructions in \cite{MR2566571} by Shvydkoy and the first author. It is also worth noting that there is no cancellation in the final solution $(u,f)$, which is the main reason that the force $f\not \in L^2_t H^{-1}$.

\subsection{Positive energy flux through each shell}
We start with constructing vector fields supported in dyadic Fourier shells with optimal energy flux. The goal is to arrange Fourier modes so that the energy flux is saturated. Since we are dealing with the NSE, intermittent flows take place of homogeneous ones comparing to \cite{MR1302409,MR2422377}.  
\begin{lemma}\label{lemma:w_n_flux}
For any dimension $d\geq 3$ and $0\leq \beta < d$, there exist $N\in\NN$ and  vector fields $w_n \in C^\infty(\TT^d)$ for $n \geq N$ with the following properties. 
\begin{enumerate}
\item The Fourier support of $w_n$ is in a shell of radius $\lambda_n$: 
\begin{equation} \label{eq:Fourier_support_w_n}
\Supp \mathcal{F}w_n \subset \big\{k: \textstyle   \lambda_n \leq |k| \leq  \lambda_{n+1} \big\};
\end{equation}
\item $L^2$ norm is normalized: $\|w_n \|_2 =1 $, and the $L^p$ scaling
\begin{equation}\label{eq:lp_scaling_wn}
\|w_n \|_p \sim_p \lambda_n^{ (\nicefrac{1}{2}-\nicefrac{1}{p}) \beta},
\end{equation}
holds for all $p \in (  1,\infty]$;
\item The energy flux through Littlewood-Paley shells satisfies
\begin{equation}\label{eq:w_n_flux}
\lim_{n\to \infty} \lambda_n^{-1-\frac{\beta}{2}}\int_{\TT^d} \D (w_n \otimes w_n )_{\leq n} \cdot (w_n )_{\leq n} \, dx=2,
\end{equation}
and $\int_{\TT^d} \D (w_n \otimes w_n )_{\leq q} \cdot (w_n )_{\leq q} =0$ when $q\neq n$.

\end{enumerate}
\end{lemma}

\begin{remark}
Note that the construction of $w_n$ does not depend on the choice of Littlewood--Paley cutoffs as long as $\Delta_q$ has frequencies supported in $[\lambda_{n-1} , \lambda_{n+1}]$.
\end{remark}
\begin{proof}[Proof of Lemma \ref{lemma:w_n_flux}]
Let $\mu_n =\lambda_n^{\frac{\beta}{d} }$. Note that $1 \leq \mu_n \leq \lambda_n$ and $\lambda_n^{-1}\mu_n \to 0$ as $n\to \infty$ for the given range of $\beta$. For $n\in \NN$, define the integer blocks $A_n$, $B_n$ and $C_n$ by
\begin{align*}
A_n &:=[ \lambda_n,  \lambda_n+2  \mu_n] \times [-\mu_n,\mu_n]^{d-1}\cap \ZZ^d\\
B_n &:= [-  \mu_n,  \mu_n]\times [ \sqrt{3}\lambda_n -4  \mu_n,  \sqrt{3}\lambda_n-2\mu_n]\times [-  \mu_n,  \mu_n]^{d-2}\cap \ZZ^d\\
C_n &:= A_n+B_n.
\end{align*}
Note that we can choose $N$ large enough so that for $n \geq N$ all the wavenumbers $k \in A_n\cup B_n\cup C_n$ satisfy
\begin{align} \label{eq:Bound_on_k_in_ABC}
\lambda_n^2 \leq |k|^2 &\leq \big[(\lambda_n+2\mu_n)^2 +(d-1)\mu_n^2\big] + \big[(\sqrt{3}\lambda_n-2\mu_n)^2 + (d-1)\mu_n^2 \big] \\
&\leq 4\lambda_n^2 = \lambda_{n+1}^2,
\end{align}
because $\lambda_n^{-1}\mu_n \to 0$ as $n\to \infty$.

Now, for $k\in \ZZ^d$, define  vector-valued functions $a_{\cdot},b_{\cdot}$ and $c_{\cdot}$ by
\begin{equation}\label{eq:a_k_b_k_c_k}
\begin{aligned}
a_k &= -\frac{1}{|A_n|^\frac{1}{2}}\Big(\Id -\frac{k\otimes k}{|k|^2}\Big) e_d\\
b_k &= \frac{1}{|B_n|^\frac{1}{2}}\Big(\Id -\frac{k\otimes k}{|k|^2}\Big)(e_1 +e_d) \\
c_k &= \frac{1}{|C_n|^\frac{1}{2}}\Big(\Id -\frac{k\otimes k}{|k|^2}\Big)e_d,
\end{aligned}
\end{equation}
where $e_1 = (1,\dots,0) $, $e_d= (0,\dots,1) $ and $|A_n|$,$|B_n|$  and $|C_n|$ denote the counting measures. 
Then define
\begin{align}
w_n: =   \sum_{k\in \ZZ^d}(i a_k\mathbf{1}_{A_n} -i a_k\mathbf{1}_{A_n^*}+  i b_k\mathbf{1}_{B_n} - i b_k\mathbf{1}_{B_n^{*}} + i c_k \mathbf{1}_{C_n} -i c_k\mathbf{1}_{C_n^*}  )e^{ik\cdot x}.
\end{align}
From \eqref{eq:a_k_b_k_c_k} it is clear that $w_n$ is divergence-free. From \eqref{eq:Bound_on_k_in_ABC} we also see that if $N$ is sufficiently large, we have the desired Fourier support property 
$$
\Supp \mathcal{F} w_n \subset \big\{k:  \textstyle  \lambda_n \leq |k| \leq  \lambda_{n+1} \big\}, \qquad n\geq N.
$$

Next, we show that $w_n$ enjoys the $L^p$ scaling \eqref{eq:lp_scaling_wn}. Estimates for all the blocks are the same, so we show them for $a_k \chi_{A_n}$ only. Using \eqref{eq:a_k_b_k_c_k}, the $L^p$ boundedness of the Leray projection, and $L^p$ estimate for the Dirichlet kernel, we get
\begin{align}
\|  i a_k\mathbf{1}_{A_n} e^{ik\cdot x} \|_p &\lesssim \frac{1}{|A_n|^{\frac12}} \|   \mathbf{1}_{A_n} e^{ik\cdot x} \|_p\\
&\lesssim \mu_n^{(\frac{1}{2}-\frac{1}{p})d}.
\end{align}
Applying such an estimate to all the other blocks, we obtain
$$
\| w_n\|_p \lesssim \mu_n^{(\frac{1}{2}-\frac{1}{p})d} = \lambda_n^{ (\frac{1}{2}-\frac{1}{p}) \beta}.
$$
It is worth noting that the implied constant depends on $c$ if $p\neq 2$. Note that $w_n$ does not have a unit $L^2$ norm, but it can be fixed in the end once the parameter $c$ is determined.

Finally let us show \eqref{eq:w_n_flux}. First, thanks to the Fourier support bound \eqref{eq:Fourier_support_w_n}, if $q>n$, then $(w_n)_{\leq q  } =w_n$. In this case, there is no flux thanks to the incompressibility. In addition, if $q<n$, then $(w_n)_{\leq q} =0$ and the flux is zero as well. We thus only need to consider $n=q$.

Denoting by $ \widehat{w}(k)$ the Fourier coefficients of $w_q$, we have
\begin{align}\label{eq:flux1_0}
\int_{\TT^d} \D (w_q \otimes w_q )_{\leq q} \cdot (w_q )_{\leq q}  = i\sum_{k_1+ k_2 +k_3 =0}  \big( \widehat{w}(k_1)\cdot k_2\big) \big( \widehat{w}(k_2)\cdot\widehat{w}(k_3)\chi_{q+1}^2(k_3)\big),  
\end{align}
where $\chi_{q+1}$ is the cutoff function we used in Section \ref{subsec:LP_def} for the Littlewood-Paley decomposition, cf. \eqref{eq:LP_chi_q}.

Since $w_q$ only has frequencies in three regions $A=A_q \cup A_q^*$, $B=B_q \cup B_q^*$, and  $C=C_q \cup C_q^*$, we only need to consider the following cases
\begin{align*}
k_1+ k_2 +k_3\in A_q+B_q+C_q^* \quad \text{and } \quad k_1+ k_2 +k_3\in A_q^*+B_q^*+C_q,
\end{align*}
as all the other combinations do not allow for $k_1+ k_2 +k_3=0$.
Thus denoting $E_q:= A_q\cup B_q\cup C_q^*$, by symmetry and \eqref{eq:flux1_0} we have 
\begin{align}\label{eq:flux1_1}
\int_{\TT^d} \D (w_q \otimes w_q )_{\leq q} \cdot (w_q )_{\leq q}  &= 2i\sum_{\substack{ k_j \in E_q \forall j  \\ k_1 +k_2+k_3 =0 }}  \big( \widehat{w}(k_1)\cdot k_2\big) \big( \widehat{w}(k_2)\cdot\widehat{w}(k_3)\chi_{q+1}^2(k_3)\big)  .
\end{align}
Notice that if $k_1 \not\in  B_q $, due to \eqref{eq:a_k_b_k_c_k}, $ \widehat{w}(k_1)$ is almost orthogonal to  $k_2$. More precisely, for all $k_j \in E_q $ with  $ k_1 +k_2+k_3 =0$,
\begin{align*}
\widehat{w}(k_1) \cdot k_2\lesssim ( \lambda_q^{-1}\mu_q) \lambda_q |\widehat{w}(k_1)| \lesssim \mu_q^{1-\frac{d}{2}}  \quad\text{whenever $k_1 \not\in  B_q $}.
\end{align*}
Using this observation, the fact that the number of combinations satisfies the estimate
\begin{align}\label{eq:lemma143}
|\{(k_1,k_2,k_3) \in \ZZ^{d}\times \ZZ^{d}\times \ZZ^{d}:  \forall j \,\, k_j \in E_q \,\,\text{and}\,\, k_1 +k_2+k_3 =0   \}|= |A_q| |B_q| \lesssim  \mu_q^{ 2d },
\end{align} 
and the bound $\widehat{w}(k_2)\cdot\widehat{w}(k_3) \lesssim |A_q|^{-\frac12} |B_q|^{-\frac12}$, we get\footnote{Here and in what follows, we write $O(Y)$ for some quantity $X=  O(Y)$ such that $|X|\leq C Y$ for some $C>0$.}  
\begin{align}\label{eq:flux1_2}
 \sum_{\substack{  k_j \in E_q, k_1 \not\in B_q\\k_1 +k_2+k_3 =0 } } \big| \big( \widehat{w}(k_1)\cdot k_2\big) \big( \widehat{w}(k_2)\cdot\widehat{w}(k_3)\chi_{q+1}^2(k_3)\big)     \big|  = O( \mu_q^{1+\frac{d}{2} }).  
\end{align}
It follows from \eqref{eq:flux1_1} and \eqref{eq:flux1_2} that
\begin{align}\label{eq:flux1_3}
\int_{\TT^d} \D (w_q \otimes w_q )_{\leq q} \cdot (w_q )_{\leq q}  &= 2i\sum_{ \substack{k_1 \in B_q, k_2\in A_q, k_3\in  C_q^* \\ k_1 +k_2+k_3 =0}  } +2i\sum_{ \substack{k_1 \in B_q,  k_2\in  C_q^*  k_3\in A_q  \\  k_1 +k_2+k_3 =0 }      } + O( \mu_q^{1+\frac{d}{2}} ).
\end{align}
By anti-symmetry,
\[
\sum_{ \substack{k_1 \in B_q, k_2\in A_q, k_3\in  C_q^* \\ k_1 +k_2+k_3 =0}  } \big( \widehat{w}(k_1)\cdot k_2\big) \big( \widehat{w}(k_2)\cdot\widehat{w}(k_3) \big) = - \sum_{ \substack{k_1 \in B_q,  k_2\in  C_q^*  k_3\in A_q  \\  k_1 +k_2+k_3 =0 }      } \big( \widehat{w}(k_1)\cdot k_2\big) \big( \widehat{w}(k_2)\cdot\widehat{w}(k_3) \big),
\]
so we need to analyze the value of $\chi_{q+1}(k_3)$. When $k_3 \in A_q$, we have  $|k_3|$ is about $\lambda_q$, and hence $\chi_{q+1}^2(k_3)$ is about $1$. When $k_3 \in C_q^*$, then $|k_3|$ is about $2\lambda_q$, and $\chi_{q+1}^2(k_3)$ is close to $0$ when $q$ is sufficiently large.
More precisely,

$$
\big| \chi_{q+1}^2(k_3)- 1 \big| \lesssim \frac{\mu_q}{\lambda_q}, \qquad k_3 \in A_q,
$$
and
$$
\big| \chi_{q+1}^2(k_3)\big| \lesssim \frac{\mu_q}{\lambda_q}, \qquad k_3 \in C_q^*,
$$
with constants depending on the Lipschitz constant of $\chi^2$. Therefore, from \eqref{eq:flux1_3} we get
\begin{equation}
\begin{split}\label{eq:lemma132}
\int_{\TT^d} \D (w_q \otimes w_q )_{\leq q} \cdot (w_q )_{\leq q}  &= -2i\Big(1+O\big(\lambda_q^{-1}\mu_q\big)\Big)\sum_{ \substack{k_1 \in B_q, k_2\in A_q, k_3\in  C_q^* \\ k_1 +k_2+k_3 =0}  } \big( \widehat{w}(k_1)\cdot k_2\big) \big( \widehat{w}(k_2)\cdot\widehat{w}(k_3) \big)\\ & \qquad    + O\big( \mu_q^{1+\frac{d}{2}} \big).
\end{split} 
\end{equation}

It remains to estimate the summation in \eqref{eq:lemma132}. Notice that for indexes in the allowed range
\begin{align*}
\widehat{w}(k_1)\cdot k_2= \frac{i\lambda_q }{|B_q|^\frac{1}{2}}(1+   O(\lambda_q^{-1}\mu_q ))   \quad \text{and}\quad   \widehat{w}(k_2)\cdot \widehat{w}(k_3) =   \frac{1 }{|A_q|^\frac{1}{2}|C_q|^\frac{1}{2}}(1+   O(\lambda_q^{-1}\mu_q )) .
\end{align*}
Therefore, we obtain
\begin{equation}\label{eq:lemma133}
\begin{split}
 \sum_{\substack{k_1 \in B_q, k_2\in A_q ,k_3\in  C_q^* \\ k_1 +k_2+k_3 =0}  } \big( \widehat{w}(k_1)\cdot k_2\big) \big( \widehat{w}(k_2)\cdot\widehat{w}(k_3) \big) &=  \frac{i\lambda_q (1+O( \lambda_q^{-1}\mu_q))^2}{(|A_q||B_q||C_q| )^{\frac{1}{2}}} \sum_{\substack{k_1 \in B_q, k_2\in A_q, k_3\in  C_q^* \\ k_1 +k_2+k_3 =0}  }  1 \\
 &=i\lambda_q (1+O( \lambda_q^{-1}\mu_q))^2|A_q|^{ \frac{1}{2}}|B_q|^{ \frac{1}{2}}|C_q|^{ -\frac{1}{2}}.
\end{split}
\end{equation}

Note that $|A_q|^{ \frac{1}{2}}|B_q|^{ \frac{1}{2}}|C_q|^{ -\frac{1}{2}} =(2  \mu_q)^{ d } (4  \mu_q)^{-\frac{d}{2}} +O( \mu_q^{\frac{d}{2}-1})$.
Then combining \eqref{eq:lemma143}, \eqref{eq:lemma132} and \eqref{eq:lemma133} we obtain
\begin{align}\label{eq:lemma144}
\int_{ \TT^d } \D (w_q \otimes w_q )_{\leq q} \cdot (w_q )_{\leq q}  &= 2(1+O( \lambda_q^{-1}\mu_q))^3 (1+O(\mu_q^{-1}) \lambda_q  \mu_q^{\frac{d}{2}}    + O((  \mu_q)^{1+\frac{d}{2}} ). 
\end{align}
Multiplying by $\lambda_q^{-1-\frac{\beta}{2}}$ and taking the limit as $q\to \infty$, we obtain \eqref{eq:w_n_flux}.
\end{proof}

\subsection{Anomalous dissipation without anomalous work}
 
With Lemma \ref{lemma:w_n_flux} in hand, we can use a simple gluing argument to prove Theorem \ref{thm:main_result_Pi_without_F}. The key is to design the life span of each frequency according to the size of the flux term $\F_q$, so that the final anomalous work $\overline{\F}=0$.

First we apply Lemma \ref{lemma:w_n_flux} with a fixed $\beta$ such that $2<\beta<  2+\nicefrac{\ep}{4}$ to obtain $w_n$. The choice of $\beta$ is dictated by the intermittency dimension $d-2$ for the energy balance (see Section 2 in \cite{1802.05785} for a discussion). Note that, in particular, these functions satisfy the $L^p$ scaling estimates listed in Lemma \ref{lemma:w_n_flux}.

Let 
\begin{align} \label{eq:Def_Lambda_n}
\Lambda_n:=   2  \int_{\TT^d} \D (w_n \otimes w_n )_{\leq n} \cdot (w_n )_{\leq n} \, dx.
\end{align}
Then there exists a sufficiently large $N \in \NN$ such that \eqref{eq:Def_Lambda_n} is positive and strictly increasing when $n \geq N$. Taking $T= \frac{1}{8}    \lambda_n^{-1-\frac{\beta}{2}} $, we aim to construct a solution $(u,f)$ on $(0,T)$.

Next, we construct time cutoffs $\chi_n$ for $n\geq N$ as follows. We start with time scales $\tau_n$ defined by
\begin{equation} \label{eq:def_tau_n}
\tau_n: = (1-2^{-1-\frac{\beta}{2}})^{-1}  4^{-1}\sum_{k\geq n} \lambda_k^{-1-\frac{\beta}{2}} = 4^{-1} \lambda_n^{-1-\frac{\beta}{2}}.
\end{equation}
Note that
\begin{equation} \label{eq:limit_of_Lambda_n_tau_n}
\lim_{n\to \infty}\Lambda_n \tau_n = 1,
\end{equation}
due to Lemma \ref{lemma:w_n_flux}.
Moreover, there exists a small constant $c_0>0$ depending on $\ep$ such that 
\begin{equation} \label{eq:def_c_0}
\tau_n > \tau_{n+1} +c_0 \tau_{n+1}^{\nicefrac{\ep}{4}} \quad \text{for all $n\geq N$.}
\end{equation}

Then we introduce smooth  cutoffs $h_n \in  C^\infty (\RR)$ as follows. First, we fix $h \in C^\infty (\RR) $ such that $h(t)=0$ when $t\leq -1$  and $h(t) =1$ when $t \geq 0$ and both $h^{\frac{1}{2}}$ and $(1-h)^{\frac12}$ are smooth. Next we define
\[
h_n(t)=h\bigg( \frac{t+1}{c_0\tau_n^{\nicefrac{\ep}{4}}} \bigg).
\]
Clearly $h_n \in  C^\infty (\RR)$ and satisfies
\begin{equation}\label{eq:lemma109}
 h_n(t) =
\begin{cases}
0,&  t \leq -1 -  c_0\tau_n^{\nicefrac{\ep}{4}}\\
 1, & t \geq - 1,
\end{cases}
\qquad |h_n'| \lesssim \tau_n^{-\nicefrac{\ep}{4}},
\end{equation}
where the constant in the derivative bound depends only on the profile $h$ and the constant $c_0>0$.

With all these in hand, for $n \geq N$ we define $\chi_n$ by  
\begin{equation}\label{eq:lemma110}
\begin{split}
\chi_n(t)&=\big[h_n((t-T)/ \tau_{n }   ) -h_{n+1}((t-T)/ \tau_{n+1})\big]^\frac{1}{2}.
\end{split}
\end{equation}
Thanks to a simple telescoping and the fact that $\tau_N = 2T$ we have $ \sum_{n\geq N } \chi_{n}^2 =1$ for $t\in (0,T)$.
By the definition, we see that
$$
\chi_n(t)= \begin{cases}
\big[h_n((t-T)/\tau_n)\big]^{\frac{1}{2}},  &t \leq T  -\tau_{n+1}-  c_0\tau_{n+1}^{1+\nicefrac{\ep}{4}} , \\
\big[1-h_{n+1}((t-T)/\tau_{n+1})\big]^{\frac{1}{2}},  &t  \geq T -\tau_{n} \\
1,  & T -\tau_{n} \leq t  \leq T  -\tau_{n+1}-  c_0\tau_{n+1}^{1+\nicefrac{\ep}{4}}
\end{cases}
$$
where we have used \eqref{eq:def_c_0} for the third line.

Now we claim that $\chi_n \in C^\infty_c(\RR)$. Indeed, if $t \geq T - \tau_{n+1} $ or $t \leq T- \tau_{n}   - c_0\tau_{n}^{1+ \nicefrac{\ep}{4}} $, then $\chi_n =0$. The smoothness of $\chi_n$ follows from that of $h^{1/2} $ and $(1 -h)^{1/2}$ since $h_n$ is defined through a translation and rescaling of  $h$. Note that to estimate the derivative of $\chi_n$ we only need to count the leading order rescaling factors from $h$ to $\chi_n$, and hence by \eqref{eq:lemma109} and \eqref{eq:lemma110}
\begin{equation} \label{eq:bound_on_chi_n'}
|\chi_n'| \lesssim \tau_{n+1}^{-1-\ep/4} .  
\end{equation}

Now we are in the position to construct the force and solution of the NSE. Let 
\begin{equation}\label{eq:construction_u2}
u := \sum_{n\geq N} \chi_n w_n \quad  \text{and}  \quad f: = \partial_t u -\Delta u + \D (u\otimes u) .
\end{equation}

In the remainder of this section, we are going to show that the solution given by \eqref{eq:construction_u2} satisfies the statement of Theorem \ref{thm:main_result_Pi_without_F}.

\begin{lemma}
The constructed solution \eqref{eq:construction_u2} satisfies the following: $(u,f) \in \mathcal{N}(Q_T)$, $\|u (t)\|_2 =1$ for all $t \in [0,T)$, and 
\begin{align}\label{eq:thmeq1}
\lim_{t\to T^-} \langle u(t) , \phi \rangle =0 \quad \text{for all $\phi  \in L^2$ }.
\end{align}
\end{lemma}
\begin{proof}
It is easy to see $u$ and $f$ are smooth on $(0,T)$. Due to disjoint Fourier supports of $w_n$ and Plancherel's formula we get
\begin{align} \label{eq:Energy_is_constant}
\|u(t) \|_2^2 = \sum_{n\geq N} \|\chi_n w_n \|_2^2  = \sum_{n\geq N}  \chi_n^2(t) =1  \quad \forall t\in [0,T).
\end{align}
Applying Plancherel's formula again we have
\begin{align*}
  \int_0^T \| \nabla u(t) \|_2^2 ~dt =  \sum_{n\geq N}\int_0^T   \|\chi_n  \nabla w_n \|_2^2 ~dt\lesssim \sum_{n\geq N} \lambda_n^2 \int_{  \Supp \chi_n} 1~dt \lesssim \sum_{n\geq N} \lambda_n^{2 } \Lambda_{n}^{-1}.
\end{align*}
Recall from \eqref{eq:Def_Lambda_n} and Lemma \ref{lemma:w_n_flux} that $\Lambda_{n} \sim \lambda_n^{\nicefrac{\beta}{2} +1} $. Since $\beta >2$, this implies that $u \in L^2 H^1$.
In addition, $\displaystyle \lim_{s\to t}\|u(s)\|_2^2$ exists (and are equal to $1$) for $t=0^+$ and $T^-$ due to \eqref{eq:Energy_is_constant}. So by Lemma~\ref{lemma:fancyLebesgue}, $(u,f) \in \mathcal{N}(Q_T)$. Finally, \eqref{eq:thmeq1} follows from the $L^p$ scaling of $w_n$ for $p<2$.
\end{proof}

This lemma  establishes that the energy has a jump as $t\to T^-$. Now we are going to show that the anomalous work vanishes.
\begin{lemma}
The anomalous work of $f$ is zero:
\begin{align*}
\overline{\F }   =   \limsup_{q \to \infty} | \F_{q} (T ,h)|=0.
\end{align*}
\end{lemma}

\begin{proof}
Thanks to the previous lemma, we can use the energy jump formula \eqref{eq:energy_jump_formula} to compute
\begin{align*}
\lim_{q \to \infty}  \left[ \Pi_q(T,h) + \F_q(T,h)\right] &= \frac{1}{2}\|u(T)\|_2^2 - \lim_{t \to T -}\frac{1}{2}\|u(t)\|_2^2\\
&= -\frac{1}{2}.
\end{align*}
So in order to show that the anomalous work is zero it suffices to prove that
\[
\lim_{q \to \infty}  \Pi_q(T,h) = -\frac{1}{2}.
\]

By definition \eqref{eq:lemma110}, the time cutoffs have disjoint support in the sense that $\chi_m \cap \chi_n = \emptyset $ if $ |m-n| >1$. Then thanks to the bound on the Fourier support of $w_n$ \eqref{eq:Fourier_support_w_n}, for each $t$ there is $n$ such that
\[
\Supp \mathcal{F} u(t) \subset \big\{k:  \textstyle  \lambda_{n-1} \leq |k| \leq  \lambda_{n+1} \big\}.
\]
So if $q>n$, then $u_{\leq q  }(t) =u(t)$, and there is no flux thanks to the incompressibility. In addition, if $q<n-1$, then $u_{\leq q}(t) =0$ and the flux is zero as well. Hence,
\begin{align}\label{eq:lemma135}
\int_{\TT^d} \D (u \otimes u )_{\leq q} \cdot (u )_{\leq q} ~dx=\chi_q^3(t) \int_{\TT^d} \D (w_q \otimes w_q )_{\leq q} \cdot (w_q )_{\leq q} \,dx+ I_{\text{Err}},
\end{align}
where the error term is defined by
\begin{align*}
I_{\text{Err}}(t) & = \int_{\TT^d} \D \Big( \sum_{ q-1 \leq l\leq q+1 } \chi_l w_l \otimes   \sum_{ q-1 \leq l\leq q+1 }\chi_{l} w_{l} \Big)_{\leq q} \cdot  \Big( \sum_{ q-1 \leq l\leq q+1 }  \chi_l w_l  \Big)_{\leq q} \,dx\\
&-\chi_q^3(t) \int_{\TT^d} \D (w_q \otimes w_q )_{\leq q} \cdot (w_q )_{\leq q} \,dx.
\end{align*}
By \eqref{eq:w_n_flux}, $w_l$ only has nonzero flux through shell $\lambda_l$. So when the index $l$ is equal to $q-1$ or $q+1$ in all three sums in the error term, the integral is zero. Hence, using the estimate \eqref{eq:lp_scaling_wn} in Lemma \ref{lemma:w_n_flux}, we have the following estimate for $I_{\text{Err}} $:
\begin{equation}\label{eq:lemma11111}
   \big| I_{\text{Err}}(t) \big| \lesssim \lambda_n^{ \frac{\beta}{2} +1}\big(\chi_{q-1}(t) + \chi_{q+1}(t)\big) \chi_q(t).
\end{equation}

Now integrating \eqref{eq:lemma135} in time and using the definition of $\Lambda_q$ \eqref{eq:Def_Lambda_n} we obtain
\begin{align}\label{eq:lemma111}
\Pi_{q}(T,h)= -\int_{T-h}^T \int_{\TT^d}   \D (u \otimes u )_{\leq q} \cdot u_{\leq q}  \,dx \,dt = -\frac{\Lambda_q}{2} \int_{T-h}^T \chi_q^3(t) \,dt + \int_{T-h}^T I_{\text{Err}}(t) \,dt.
\end{align}
Recall that our goal is to prove that $\Pi_{q}(T,h) \to -\frac{1}{2} $. To this end, due to \eqref{eq:lemma111}, it suffices to show
$$
\Lambda_q \int_{T-h}^T \chi_q^3 \, dt \to 1 \quad \text{and} \quad \int_{T-h}^T I_{\text{Err}}(t) \,dt\to 0.
$$

First notice that
\[
\lim_{q\to\infty}\tau_q^{-1} \big| \{t : \chi_q(t) =1   \} \big| =1,
\]
and hence
\[
\lim_{q\to\infty} \Lambda_q \int_{\chi_q =1 } \chi_q^3 \, dt =1
\]
due to \eqref{eq:limit_of_Lambda_n_tau_n}.
In addition, from \eqref{eq:lemma109} and \eqref{eq:lemma110} it follows that
\begin{equation*}
\big| \{t : 0< \chi_q(t) < 1   \} \big| \leq  2\tau_q^{1+\nicefrac{\ep}{4}} ,
\end{equation*}
so for $q$ large enough,
\begin{align*}
  \Lambda_q \int_{T-h}^T \chi_q^3 \, dt =  \Lambda_q \int_{\chi_q =1 } \chi_q^3 \, dt + \Lambda_q\int_{0 < \chi_q <1 } \chi_q^3 \, dt \to 1, \quad \text{as $q \to \infty$,}
\end{align*}
as the second term converges to zero thanks to \eqref{eq:limit_of_Lambda_n_tau_n}.

By \eqref{eq:lemma11111}, a similar argument shows that
$$
\left| \int_{T-h}^T I_{\text{Err}}(t) \,dt \right| \lesssim \lambda_q^{\frac{\beta}{2} +1} \int_{\Supp \chi_q \cap \Supp \chi_{q-1}}  1 \,dt  \lesssim \lambda_q^{\frac{\beta}{2} +1} \tau_q^{1+\nicefrac{\ep}{4}} \to 0 \quad \text{as $q \to \infty$.}
$$
\end{proof}

At last, we verify the functional classes for $u$ and $f$, concluding the proof of Theorem \ref{thm:main_result_Pi_without_F}.
\begin{lemma} 
The solution $u$ is almost Onsager critical: $u \in L_t^{3 } B^{\frac{1}{3} -\ep }_{3,\infty} \cap L^{p}_t L^{q} $ for any $p$, $q$ with $\frac{2}{p} +\frac{2}{q} =1+\ep $, and the force $f \in L^{2-\ep} H^{-1}$.
\end{lemma}
\begin{proof}
Thanks to \eqref{eq:lp_scaling_wn} in Lemma \ref{lemma:w_n_flux} and \eqref{eq:lemma110}, we can compute 
\begin{align*}
\int \|u (t) \|_{B^{\frac{1}{3} -\ep}_{3,\infty}}^3 \, dt\leq \sum_{n \geq N} |\Supp \chi_n | \| w_n\|_{B^{\frac{1}{3} -\ep}_{3,\infty}}^3 \lesssim \sum_{n \geq N} (\tau_n -\tau_{n+1}) \big[\lambda_{n}^{ \frac{1}{3} -\ep } \lambda_{n}^{ \frac{\beta}{6}   }   \big]^3.
\end{align*}
Since $\tau_n \sim \Lambda_n^{-1} \sim \lambda_n^{-\nicefrac{\beta}{2}-1}  $, the summation is indeed finite.

To verify the membership in $ L^{p}_t L^{q}$ we use \eqref{eq:lp_scaling_wn} once again to obtain
\begin{align}
\int \|u (t) \|_{q}^p  \, dt\lesssim  \sum_{n} \big|\Supp \chi_n   \big| \|w_n   \|_{q}^p \sim \sum_n \lambda_n^{\frac{\beta}{2} -1- \frac{ \beta p}{2} \ep}  .
\end{align}
Recall that $2<\beta<  2+\nicefrac{\ep}{4}$. Noticing that the powers obey $  \frac{\beta}{2} -1- \frac{ \beta p}{2} \ep < (\frac{1}{8} - \frac{ \beta p}{2} ) \ep <0 $, we conclude $ u\in L^{p}_t L^{q}$ for any $\frac{2}{p} +\frac{2}{q} =1+\ep $.

To estimate the force $f$, we just compute each part separately:
\begin{align}
\int \|f (t )\|_{H^{-1}}^{2-2\ep} \, dt \lesssim \int \left( \|\partial_t u \|_{H^{-1}}^{2-2\ep}   +  \|\nabla u \|_{ 2}^{2-2\ep}  + \|u\otimes u \|_2^{2-2\ep}  \right) \, dt.
\end{align}
where we have used the fact $|\nabla|^{-1}\D$ is $L^2 \to L^2$ bounded.
Since $u\in L^2 H^1$, the second term is under control. For the nonlinear part,  we again obtain 
\begin{align}\label{eq:lemma1113}
\int   \|u\otimes u \|_2^{2-2\ep}  \, dt \lesssim \int   \|u  \|_4^{4-4\ep}  \, dt \lesssim \int   \sum_n |\Supp \chi_n |  \|w_n  \|_4^{4-4\ep} \, dt,
\end{align}
where we have used the fact that $\Supp_t \chi_n  $ have only finite overlaps. Using the $L^p$ scaling in Lemma \ref{lemma:w_n_flux}, the desired bound follows from \eqref{eq:lemma1113}, the fact that $\tau_n \sim \lambda_n^{-\nicefrac{\beta}{2}-1}  $, and $\beta<  2+\nicefrac{\ep}{4}$:
\begin{align}
\int   \|u\otimes u \|_2^{2-2\ep}  \, dt \lesssim  \sum_n  \tau_n   \lambda_{n}^{ \frac{\beta(4-4\ep)}{4}   }   <\infty.
\end{align}
At last, we check the time derivative. From the Fourier support of $w_n$ and the temporal support of $\chi_n$ it follows that
\begin{align*}
\int   \|\partial_t  u \|_{H^{-1}}^{2-2\ep}  \, dt & \lesssim \sum_{n} \| w_n \|_{H^{-1}}^{2-2\ep} \int  |\chi_n'|^{2-2\ep} \, dt \\
&\lesssim \sum_{n} \lambda_n^{-2+2\ep} \int  | \chi_n'|^{2-2\ep} \, dt.
\end{align*}
Using bound \eqref{eq:bound_on_chi_n'}, $\tau_n \sim \Lambda_n^{-1}$ due to \eqref{eq:def_tau_n}, and $\Lambda_{n} \sim \lambda_n^{\nicefrac{\beta}{2} +1} $ from Lemma \ref{lemma:w_n_flux}, we obtain
$$
|\chi_n'| \lesssim \tau_{n+1}^{-1-\nicefrac{\ep}{4}}\sim \Lambda_{n+1}^{1+\nicefrac{\ep}{4}}  \sim \lambda_{n+1}^{(1+\nicefrac{\ep}{4})(\nicefrac{\beta}{2} +1) },
$$ and conclude that
\begin{align*}
\int   \|\partial_t  u \|_{H^{-1}}^{2-2\ep}  \, dt  
&\lesssim \sum_{n} \lambda_n^{-2+2\ep}  \lambda_{n+1}^{(1-\nicefrac{3\ep}{2}-\nicefrac{\ep^2}{2})(\nicefrac{\beta}{2} + 1)} \lesssim \sum_{n} \lambda_n^{- 2\ep} <\infty .
\end{align*}
\end{proof}
\section{Violating energy balance with continuous energy}\label{sec:continuousE}

This section is devoted to the proof of Theorem \ref{thm:main_result_continuous_E_with_Pi}. Comparing with the example in Section \ref{sec:zeroF_nonzeroPi}, here the energy does not completely transfer to the next shell. Instead, the force also injects energy into larger and larger shells when $t \to T^-$. Towards this end, we first construct intermittent vector fields with large flux, similar to Lemma \ref{lemma:w_n_flux}, and then apply a time-dependent wavenumber cutoff to obtain the final solution. These vector fields can be viewed as intermittent versions of the example in \cite{MR1302409}.

\subsection{Positive energy flux through each sphere}
\begin{lemma}\label{lemma:w_flux}
For any integer $d\geq 3$, there exist  constants $C>0$ and $N\in\NN$ such that for any $0\leq \beta < d$, there exists a divergence free vector fields $w  \in L^2(\TT^d)$  with the following properties. 
\begin{enumerate}
\item For any $q\in \NN$, the bound holds
\begin{equation}\label{eq:lp_scaling_w}
\|w_q \|_p \lesssim_p  \lambda_q^{ (\nicefrac{1}{2}-\nicefrac{1}{p}) \beta}, 
\end{equation}
for all $1<p \leq \infty$;
\item For any $q\geq N$, the energy flux through wavenumber $\lambda_q$ satisfies 
\begin{equation}
\int_{\TT^d} \D (w  \otimes w  )_{\leq q} \cdot (w  )_{\leq q} \, dx \geq  C \lambda_q^{\nicefrac{\beta }{2} + 1}.
\end{equation}
 
\end{enumerate}
\end{lemma}
\begin{proof}[Proof of Lemma \ref{lemma:w_flux}]
First, we introduce a small parameter $ \ep:= \frac{d-\beta}{100d}$. Let $\mu_q =\lambda_q^{\frac{\beta+6\ep}{d} }$. Note that $1 \leq \mu_q \leq \lambda_q$ for the given range of $\beta$ and $\ep$. Let 
\begin{align}\label{eq:k_i}
\xi_1=(1,0,0,\dots,0)  \quad \xi_2= (0,1,0\dots,0)\quad \xi_3= (1,1,0,\dots,0)\quad \xi_4=(1,-1,0,\dots,0)
\end{align}
and
\begin{align}\label{eq:e_i}
e_1=(0,0,-1,0,\dots,0)  \quad e_2= (1,0, 1,0,\dots,0)\quad e_3= (0,0,1,0,\dots,0)\quad e_4=(1,1,-1,0,\dots,0).
\end{align}
Note that $\xi_j \cdot e_j =0$, which is needed for the divergence free condition, as $e_j$ will be the direction of the Fourier coefficient $\widehat{w}(k)$ for $k\sim \xi_j$.

Define the set
\begin{equation}\label{eq:Omeg_i}
\Omega_{q,j} = \{ \lambda_q \xi_j + \kappa:  \kappa \in[- \mu_q  , \mu_q ]^d \cap \ZZ^d \},
\end{equation}
and $\Omega_{q,j}^* = -\Omega_{q,j}$. Note that $ \Omega_{q,i} \subset\ZZ^d$ and  for all $k_j \in \Omega_{p,j}$ 
\begin{equation}\label{eq:smallangles}
\begin{aligned}
|k_j - \xi_j| &=O(\mu_q).
\end{aligned}
\end{equation}

For $q\in \ZZ$, $1\leq j \leq 4$, define vector-valued functions $a_{q,j}: \ZZ^d \to \mathbb{C}^d$ by
\begin{equation}\label{eq:a_qkj}
a_{q, j}(k) =
\begin{cases} 
  -i\frac{\lambda_q^{-\ep}}{|\Omega_{q,j}|^\frac{1}{2}}\Big(\Id -\frac{k\otimes k}{|k|^2}\Big) e_j& \text{if } k\in \Omega_{q,j}  \\
 i\frac{\lambda_q^{-\ep}}{|\Omega_{q,j}|^\frac{1}{2}}\Big(\Id -\frac{k\otimes k}{|k|^2}\Big) e_j& \text{if } k\in \Omega_{q,j}^* \\
0& \text{otherwise},  
\end{cases}
\end{equation}
where $|\Omega_{q,j}|$ denote the counting measure of $\Omega_{q,j}$. Since for each $j$, the set $\Omega_{q,j}$ has the same size, we will simply write $|\Omega| $ in what follows. Also note that due to the Leray projection $\Id -\frac{k\otimes k}{|k|^2} $, each $a_{q, j}$ satisfies $a_{q, j}(k) \cdot k =0 $.

With coefficients $a_{q,j}$ in hand, we define the vector field $w$ on the Fourier side as
\begin{align}
w : = \sum_{ q \geq 1}  \sum_{1 \leq j \leq 4} \sum_{k \in \ZZ^d }a_{q,j}(k)e^{i k\cdot x} = 
\sum_{\kappa \in \ZZ^d} \widehat{w}(\kappa) e^{i\kappa\cdot x},
\end{align}
where the Fourier coefficient $\widehat{w}(\kappa)$ satisfies
\[
\widehat{w}(\kappa) = \sum_{ q \geq 1}  \sum_{1 \leq j \leq 4} a_{q,j}( \kappa).
\]
From \eqref{eq:a_qkj} it is clear that $w $ is real-valued and divergence-free. The $L^p$ scaling of $w_q$ for $q \in \NN$ is obtained along the same lines as in the proof of Lemma \ref{lemma:w_n_flux}, and thus omitted. In particular, $w \in L^2$ thanks to the decay factor $\lambda_q^{-\ep}$ in \eqref{eq:a_qkj}.

Finally, let us compute the energy flux.
Given that there are no interactions between different shells, on the Fourier side we have
\begin{align*}
\int_{\TT^d} \D (w  \otimes w  )_{\leq q} \cdot (w  )_{\leq q} & = i\sum_{\substack{k_1, k_2, k_3 \in \ZZ^d\\k_1+ k_2 +k_3 =0}}  \big( \widehat{w}(k_1)\cdot k_2\big) \big( \widehat{w}(k_2)\cdot\widehat{w}(k_3)\chi_{q+1}^2(k_3)\big)\\
&=i\sum_{\substack{k_1,k_2, k_3 \in \cup_j (\Omega_{q,j}\cup \Omega^*_{q,j}) \\ k_1+ k_2 +k_3 =0}}  \big( \widehat{w}(k_1)\cdot k_2\big) \big( \widehat{w}( k_2)\cdot\widehat{w}( k_3)\chi_{q+1}^2( k_3)\big)
\end{align*}
We claim that this sum is equal to
\begin{equation}
\begin{aligned}\label{eq:fluxnonzero1}
i\sum_{ k_1+ k_2 +k_3 =0}  &\big( \widehat{w}( k_1)\cdot k_2 \big) \big( \widehat{w}( k_2)\cdot\widehat{w}( k_3)\chi_{q+1}^2( k_3)\big)\\
&  = i\sum_{k_1+k_2+k_3=0}  \big( \widehat{w}(\lambda_q \overline{ \xi}_{k_1})\cdot \overline{ \xi}_{k_2}\lambda_q\big) \big( \widehat{w}(\lambda_q\overline{ \xi}_{k_2})\cdot\widehat{w}(\lambda_q \overline{\xi}_{k_3})\chi_{q+1}^2(\lambda_q\overline{ \xi}_{k_3})\big) +O( \mu_q|\Omega|^{\frac{1}{2}}),
\end{aligned} 
\end{equation}
where we define $\overline{ \xi}_k = \xi_j$ if $k \in \Omega_{q,j}$ and $\overline{ \xi}_k =- \xi_j$ if $k\in \Omega_{q,j}^*$, so $\overline{\xi}_k$ is in the center of the cube that $k$ belongs to, and is equal to the corresponding $\xi_j$ in \eqref{eq:k_i}.
Indeed, since the number of nontrivial interactions is bounded by $|\Omega_{q}|^2$, the claim follows  from \eqref{eq:smallangles}, \eqref{eq:a_qkj}, and the continuity of the dot product and Leray projection. 
 
Following the argument in \cite[p.13]{MR1302409},  the only nonzero terms in \eqref{eq:fluxnonzero1} are as follows. For brevity of the notation, we simply write $\Omega_j =\Omega_{q,j} $ and $\Omega_j^* = \Omega_{q,j}^* $.

For $ \xi_3-\xi_1-\xi_2=0 $, the first group 
\begin{equation}\label{eq:interaction1}
\begin{aligned}
&2i\sum_{k_1 \in \Omega_2,\,  k_2\in \Omega_1 ,\,  k_3\in \Omega_3^* } \lambda_q \big( \widehat{w}(\lambda_q \xi_2)\cdot \xi_1\big) \big( \widehat{w}(\lambda_q\xi_1)\cdot\widehat{w}(-\lambda_q\xi_3)\big)  (\chi_{q+1}^2(\lambda_q \xi_3) -\chi_{q+1}^2(\lambda_q \xi_1))\\
=&2\sum_{k_1 \in \Omega_2,\,  k_2\in \Omega_1 ,\,  k_3\in \Omega_3^* }  \lambda_q^{1-3\ep}| \Omega_q|^{- \frac{3}{2}}(\chi^2(1/2 e_1) -\chi^2(\sqrt{2}/2e_1)) \gtrsim \sum_{k_1 \in \Omega_2,\,  k_2\in \Omega_1 ,\,  k_3\in \Omega_3^* }  \lambda_q \mu_q^{-\frac{3d}{2}} .
\end{aligned} 
\end{equation}
For $ \xi_4-\xi_1+\xi_2 =0$, the second group  
\begin{equation}\label{eq:interaction2}
\begin{aligned}
&2i\sum_{k_1 \in \Omega_2,\,  k_2\in \Omega_4 ,\,  k_3\in \Omega_1^* }\lambda_q  \big( \widehat{w}(\lambda_q\xi_2)\cdot \xi_4\big) \big( \widehat{w}(\lambda_q\xi_4)\cdot\widehat{w}(-\lambda_q\xi_1)\big)  (\chi_{q+1}^2(\lambda_q \xi_1) -\chi_{q+1}^2(\lambda_q \xi_4))\\
=&2 \sum_{k_1 \in \Omega_2,\,  k_2\in \Omega_4 ,\,  k_3\in \Omega_1^* }  \lambda_q^{1-3\ep}   |\Omega_q|^{-\frac{3}{2}}(\chi^2(1/2e_1) -\chi^2(\sqrt{2}/2e_1))\gtrsim \sum_{k_1 \in \Omega_2,\,  k_2\in \Omega_4 ,\,  k_3\in \Omega_1^* }  \lambda_q \mu_q^{-\frac{3d}{2}} 
\end{aligned} 
\end{equation}
and finally for $2\xi_1 -\xi_3 -\xi_4=0 $ and $2\xi_2 -\xi_3 +\xi_4=0 $, the third group
\begin{equation}\label{eq:interaction3}
\begin{aligned}
&2i\sum_{k_1 \in \Omega_4^*,\,  k_2\in \Omega_1 ,\,  k_3\in \Omega_3^* } \lambda_{q+1} \big( \widehat{w}(-\lambda_q\xi_4)\cdot \xi_1\big) \big( \widehat{w}(\lambda_{q+1}\xi_1)\cdot\widehat{w}(-\lambda_q\xi_3)\big)  (\chi_{q+1}^2(\lambda_q \xi_1) -\chi_{q+1}^2(\lambda_q \xi_4))\\
=&2\sum_{k_1 \in \Omega_4^*,\,  k_2\in \Omega_1 ,\,  k_3\in \Omega_3^* }  \lambda_{q+1}^{1- \ep}       \lambda_{q}^{-2\ep} |\Omega_q|^{-1}|\Omega_{q+1}|^{-\frac{1}{2}}(\chi^2(\sqrt{2}/2e_1)  -\chi^2(e_1)) \\
&\gtrsim \sum_{k_1 \in \Omega_4^*,\,  k_2\in \Omega_1 ,\,  k_3\in \Omega_3^* }\lambda_q \mu_q^{-\frac{3d}{2}} ,
\end{aligned} 
\end{equation}
and respectively the forth group
\begin{equation}\label{eq:interaction4}
\begin{aligned}
&2i\sum_{k_1 \in \Omega_4 ,\,  k_2\in \Omega_2 ,\,  k_3\in \Omega_3^* } \lambda_{q+1} \big( \widehat{w}(\lambda_q\xi_4)\cdot \xi_2\big) \big( \widehat{w}(\lambda_{q+1}\xi_2)\cdot\widehat{w}(-\lambda_q\xi_3)\big)  (\chi_{q+1}^2(\lambda_q \xi_3) -\chi_{q+1}^2(\lambda_{q+1} \xi_2))\\
=&2\sum_{k_1 \in \Omega_4 ,\,  k_2\in \Omega_2 ,\,  k_3\in \Omega_3^* }  \lambda_{q+1}^{1- \ep}       \lambda_{q}^{-2\ep} |\Omega_q|^{-1}|\Omega_{q+1}|^{-\frac{1}{2}}(\chi^2(\sqrt{2}/2e_1)  -\chi^2(e_1))\\
&\gtrsim  \sum_{k_1 \in \Omega_4 ,\,  k_2\in \Omega_2 ,\,  k_3\in \Omega_3^* }\lambda_q \mu_q^{-\frac{3d}{2}} ,
\end{aligned} 
\end{equation}

Now it remains to count the number of interactions in \eqref{eq:interaction1}--\eqref{eq:interaction4}. A very rough lower bound suffices due to the positivity of the summands. By considering two cubes with half the length, there are at least $( \frac{1}{2} \lfloor \mu \rfloor )^d \times ( \frac{1}{2} \lfloor \mu \rfloor )^d   $ many interactions. Therefore, we have
\begin{align}
\int_{\TT^d} \D (w  \otimes w  )_{\leq q} \cdot (w  )_{\leq q} \, dx & \geq C \lambda_q \mu_q^{\frac{d}{2}}    +O(\mu_q|\Omega|^{\frac{1}{2}}),
\end{align}
where $C$ is an absolute constant. The conclusion follows from the fact that $\lim_{q \to \infty } \mu_q \lambda_q^{-1}  = 0  $.
\end{proof}

\subsection{Anomalous dissipation with continuous energy}
We proceed to construct the solution using the vector field $w$ in Lemma \ref{lemma:w_flux}. Here we use a different type of argument from the previous section. We simply use a time-dependent wavenumber to apply frequency cutoff.

Let $\chi \in C^\infty_c(\RR^d)$ be the cutoff function in defining the Littlewood-Paley decomposition. We introduce a time-dependent wavenumber  by
\begin{equation}\label{eq:def_Lambda}
\Lambda(t): = (T-t)^{-\frac{1-\ep}{2}},
\end{equation}
where $\ep>0$ can be arbitrary.
Let $P_t$ be the projection into frequencies $\lesssim \Lambda(t) $ by a multiplier with symbol $ \chi(\xi \Lambda(t)^{-1})$. 
Applying Lemma \ref{lemma:w_flux} with $2<\beta <2 +\frac{\ep}{4}$, we define the solution and the force by
\begin{equation}\label{eq:def_continousE_solution}
u := P_t w \quad \text{and }\quad f :=\partial u -\Delta u + \D(u\otimes u).
\end{equation}

We first show that the solution $u$ verifies all properties stated in Theorem \ref{thm:main_result_continuous_E_with_Pi}.

\begin{lemma}
For any $\ep>0$, the solution given by \eqref{eq:def_continousE_solution} is in the class $\mathcal{N}(Q_T)$  and $u \in C([0,T];L^2)$. 
\end{lemma}
\begin{proof}
Let us show $(u,f) \in \mathcal{N}(Q_T)$. Smoothness of $u$ and $f$ on $(0,T)$ follows directly from the compactness of Fourier support. Thus $(u,f)$ solves the NSE in classical sense. In view of Lemma \ref{lemma:fancyLebesgue}, to prove that $\langle u, f \rangle \in \mathcal{L}(0,T)$ it suffices to show 
$$
u \in C([0,T]; L^2) \cap L^2(0,T;H^1). 
$$ 
For any $0\leq s \leq t \leq T$, by Plancherel's formula we have
\begin{align}
\|u(t) - u(s) \|_2^2  \lesssim \sum_{s^{-1-\ep} \lesssim 2^q \lesssim t^{-1-\ep}} \|w_q \|_2^2,  
\end{align}
which together with $w \in L^2$ implies $ \|u(t) - u(s) \|_2 \to 0$ as $s \to t$. So we get $u \in C([0,T]; L^2) $. To show the finite dissipation, using Plancherel's formula  again gives
\begin{align}
\int_0^T \| \nabla u \|_2^2 \, dt \lesssim \int_0^T \Lambda(t)^2 \, dt = \int_0^T (T-t)^{-1+\ep} \, dt<\infty.
\end{align}
Therefore, by Lemma \ref{lemma:fancyLebesgue} it follows that $\langle  u,f \rangle \in \mathcal{L}(0,T)$.

\end{proof}

\begin{lemma}
The anomalous dissipation of $u$ is nonzero:
\begin{align*}
\overline{\Pi}  =   \lim_{q \to \infty}   \Pi_{q} (T ,h)  >0.
\end{align*}
\end{lemma}
\begin{proof}
Thank to Lemma \ref{lemma:welldefined} we may compute $\overline{\Pi}$ with $h=\frac{T}{2}$. Since by definition $u = P_t w$, we have
\begin{align*}
  \int^{T}_{T/2}\int_{\TT^d} \D (u \otimes u )_{\leq q} \cdot u_{\leq q}  = \int^{T}_{T/2} i\sum_{k_1+ k_2 +k_3 =0}  \big( \widehat{w}(k_1)\cdot k_2\big) \big( \widehat{w}(k_2)\cdot\widehat{w}(k_3)\chi_{q+1}^2(k_3)\big) \chi_t(k_1,k_2,k_3) \,dt,
\end{align*}
where $\chi_t(k_1,k_2,k_3) = \chi(k_1 /\Lambda)\chi(k_2 /\Lambda)\chi(k_3 /\Lambda) $. From the proof of Lemma \ref{lemma:w_flux}, we see that $k_i \sim \lambda_q$ for all nonzero interactions. Using the definition of $\Lambda(t)$ \eqref{eq:def_Lambda} and the fact that $\beta <2 +\frac{\ep}{4}$, we obtain the bound
\begin{align}
\int^{T}_{T/2} \chi_t(k_1,k_2,k_3) \,dt\sim \lambda_q^{-\frac{1+\ep}{1-\ep}-1} \geq \lambda_q^{ -\frac{\beta}{2} -1}. 
\end{align} 
On the other hand, by Lemma \ref{lemma:w_flux},
\[
\begin{split}
 i\sum_{k_1+ k_2 +k_3 =0}  \big( \widehat{w}(k_1)\cdot k_2\big) \big( \widehat{w}(k_2)\cdot\widehat{w}(k_3)\chi_{q+1}^2(k_3)\big) &= \int_{\TT^d} \D (w  \otimes w  )_{\leq q} \cdot (w  )_{\leq q}\\
 &\geq C \lambda_q^{\frac{\beta}{2}+1}.
\end{split}
\]
Then by Fubini's theorem we have
$$
\Pi_{q} (T ,T/2)= \int^{T}_{T/2}\int_{\TT^d} \D (u \otimes u )_{\leq q} \cdot u_{\leq q}  \gtrsim \lambda_q^{ -\frac{\beta}{2} -1} \cdot \lambda_q^{\frac{\beta}{2}+1} =1.
$$

\end{proof}

The rest of this section is devoted to proving the stated regularity of $u$ and $f$. We start with estimating $u$ in spaces near $ L_t^{3 } B^{\frac{1}{3}  }_{3,\infty}$ and near $L^{p}_t L^{q}_x$ for $\frac{2}{p} +\frac{2}{q} =1 $.
\begin{lemma}
The solution $u$ is almost Onsager critical: $u \in L_t^{3 } B^{\frac{1}{3} -\ep }_{3,\infty}$ and $u \in L^{p}_t L^{q}_x$ for any $\frac{2}{p} +\frac{2}{q} =1+\ep $.
\end{lemma}

\begin{proof}
First, recall that $\Lambda \in L^2$ due to \eqref{eq:def_Lambda}.
By definition of the Besov space $B^{\frac{1}{3} -\ep }_{3,\infty} $  and \eqref{eq:lp_scaling_w}, we have
\begin{align}
\int_0^T \|u \|_{B^{\frac{1}{3} -\ep }_{3,\infty}}^3 \, dt \lesssim  \int_0^T \L^{\frac{\beta}{2} + 1 -3\ep } \, dt ,
\end{align}
which is indeed finite thanks to $\beta <2 +\frac{\ep}{4}$ and hence $\frac{\beta}{2}+1-3\ep <2$. Similarly, we also obtain
\begin{align}
\int_0^T \|u \|_{q}^p \, dt \lesssim  \int_0^T \L^{\beta(\frac{p}{2}- \frac{p}{q}) } \, dt.
\end{align}
Using $\frac{2}{p} +\frac{2}{q} =1+\ep $ and again $\beta <2 +\frac{\ep}{4}$, it is not hard to show $\beta(\frac{p}{2}- \frac{p}{q})  <2$, which together with $\Lambda \in L^2$ implies 
$$
\int_0^T \|u \|_{q}^p \, dt <\infty.
$$ 
\end{proof}

Next, we estimate the forcing $f$. To do so let us introduce a general lemma for the projection $P_t$.

\begin{lemma}\label{lemma:timeder_bound}
Suppose $g(t)$ is a smooth function on $[0,T)$ such that $ g \geq 1$. Let $P_g$ be the frequency localized operator with symbol $\chi( \frac{\xi}{g (t)} )$. For any $ u \in L^2(\TT^d)$ and $s\in \RR$, there holds
\begin{equation}
\| \partial_t P_g u (t)\|_{H^s} \lesssim_s \frac{|g'| }{|g| }  |g|^{s},
\end{equation}
where $P_g u (t)  $ denotes the evaluation of $P_g u$ at time $t$.
\end{lemma}
\begin{proof}
By chain rule, on Fourier side we have
\begin{align*}
\partial_t    \widehat{P_g u}   =      - \frac{\nabla\chi(g^{-1} \X)  \cdot \X g'}{g^{2}} \widehat{u}  .
\end{align*}
Due to the choice of cutoff function $\chi$, we see that the multiplier $\partial_t    \widehat{P_g}$ satisfies
\begin{align*}
 \Supp   \frac{\nabla\chi (g^{-1} \X) \cdot \X g'}{g^{2}}   \subset \{\X   : \textstyle \frac{1}{2}g\leq  |\X| \leq  g \},
\end{align*}
and consequently
\[
\left| \frac{\nabla\chi (g^{-1} \X) \cdot \X g'}{g^{2}} \right| \lesssim \frac{|g'| }{|g| }.
\]
The result follows from Plancherel's formula.

\end{proof}

\begin{lemma}
The forcing $f$ defined in \eqref{eq:def_continousE_solution} satisfies 
\begin{align*}
f \in     L^{2-\ep}_t H^{-1 } .
\end{align*}
\end{lemma}
\begin{proof}
As in Theorem \ref{thm:main_result_Pi_without_F}, we just compute each part of the force separately
\begin{align*}
\int_0^T \|f (t )\|_{H^{-1}}^{2-2\ep} \, dt \lesssim \int_0^T \left(\|\partial_t u \|_{H^{-1}}^{2-2\ep}   +  \|\nabla u \|_{ 2}^{2-2\ep}  + \|u\otimes u \|_2^{2-2\ep}  \right)\, dt.
\end{align*}
By the same argument as in the proof of Theorem~\ref{thm:main_result_Pi_without_F}, we have $\int \left(\|\nabla u \|_{ 2}^{2-2\ep}  + \|u\otimes u \|_2^{2-2\ep} \right) \, dt <\infty$. Thus we only focus on the time derivative part. Applying Lemma \ref{lemma:timeder_bound} with $g(t)= \Lambda(t)$ and using \eqref{eq:def_Lambda}, we immediately get
\begin{align*}
 \int_0^T \|\partial_t u \|_{H^{-1}}^{2-2\ep}\,  dt \lesssim  \int_0^T     (T-t    )^{(-\frac{1}{2} -\frac{\ep}{2} )(2-2\ep)} \, dt \lesssim  \int_0^T     (T-t    )^{-1+ \ep^2} \, dt < \infty.
\end{align*}

\end{proof}

\section*{Acknowledgement}

The authors were partially supported by the NSF grant DMS--1517583. We thank the anonymous referees for many helpful comments which greatly improve the quality of the paper.

\bibliographystyle{alpha}
\bibliography{f_nse} 

\end{document}